\documentclass[11pt,leqno,amscd,amssymb]{amsart}
\usepackage[english]{babel}
\usepackage{amsfonts,amssymb,graphicx,amsmath,amsthm,hyperref,fullpage,tikz,amscd,mathrsfs}\usepackage{enumitem}
\usetikzlibrary{matrix}
\usetikzlibrary{cd}

\usepackage{mathtools}
\allowdisplaybreaks

\newcommand{\Hom}{\text{\rm Hom}}

\newcommand{\Ext}{\operatorname{Ext}}

\newcommand{\cH}{\mathcal{H}}
\newcommand{\N}{{\mathbb N}}

\newcommand{\Z}{\mathbb Z}

\newcommand{\pde}{\mathcal P^d_{q,e}}

\newcommand{\td}{^{\otimes d}}
\newcommand{\te}{^{\otimes e}}

\newtheorem{thm}{Theorem}[section] 

\newtheorem{lem}[thm]{Lemma}
\newtheorem{cor}[thm]{Corollary}

\newtheorem{prop}[thm]{Proposition}
\theoremstyle{definition}

\newtheorem{defn}[thm]{Definition}

\newtheorem{rem}[thm]{Remark}

\newcommand{\inj}{\hookrightarrow}

\newcommand{\Id}{\operatorname{Id}}
\newcommand{\GL}{\operatorname{GL}}

\newcommand{\uqgln}{U_q(\mathfrak{gl}_n)}

\newcommand{\pd}{\mathcal P^d_{q}}

\begin{document}

\title{Quantum polynomial functors from $e$-Hecke pairs}
\author{Valentin Buciumas}
\address{Einstein Institute of Mathematics, Edmond J. Safra Campus, Givat Ram, The Hebrew
University of Jerusalem, Jerusalem, 91904, Israel}
\email{valentin.buciumas@gmail.com}
\author{Hankyung Ko}
\address{Mathematisches Institut, Universit{\"a}t Bonn, Endenicher Allee 60,
53115 Bonn, Germany}
\email{hankyung@math.uni-bonn.de}
\date{}
\maketitle

\begin{abstract}
We define a new category of quantum polynomial functors extending the quantum polynomials introduced by Hong and Yacobi. We show that our category has many properties of the category of Hong and Yacobi and is the natural setting in which one can define composition of quantum polynomial functors. Throughout the paper we highlight several key differences between the theory of classical and quantum polynomial functors.  
\end{abstract}

%\tableofcontents

\section{Introduction}

Hong and Yacobi \cite{HongYacobi} introduced a category of quantum polynomial functors which quantizes the strict polynomial functors of Friedlander and Suslin \cite{FS}. The purpose of this paper is to introduce higher level categories of quantum polynomial functors that extend the construction of Hong and Yacobi and explain why they give a natural quantization of classical polynomial functors. The most visible advantage of our definition is that we are now able to compose quantum polynomial functors.

%In order to explain why our quantization is natural, let us first describe the classical polynomial functors. % of Friedlander and Suslin \cite{FS}. % and became an important tool in the study of rational cohomology of algebraic groups.
A polynomial functor is defined as a functor between vector spaces which is polynomial on the space of morphisms. The definition can be formulated as follows. Let $k$ be a field, and let $\mathcal{V}$ be the category of finite dimensional vector spaces over $k$. The $d$-th symmetric group $S_d$ acts on $V^{\otimes d}$ by permuting tensor factors. Let $\Gamma^d \mathcal{V}$ be the category which has the same objects as $\mathcal{V}$ does while the set of morphisms between $V, W \in \mathcal{V}$ is
\begin{equation}\label{eq:homsp}
\text{Hom}_{\Gamma^d\mathcal{V}}(V, W) = \Gamma^d\text{Hom}(V,W) =(\text{Hom}(V, W)^{\otimes d})^{S_d}.  
\end{equation}
The category $\mathcal{P}^d$ of polynomial functors of degree $d$ is the category of linear functors $F: \Gamma^d \mathcal{V} \to \mathcal{V}$. A polynomial functor is by definition an object in the category $\mathcal P=\bigoplus_{d\geq 0} \mathcal{P}^d$.

The category $\mathcal{P}$ has initially been introduced by \cite{FS} to prove the cohomological finite generation of finite group schemes over a field, where they use certain rational cohomology computations for the general linear groups effectively done in the category of polynomial functors.
Another example of important rational cohomology result obtained from the polynomial functors is the untwisting of Frobenius due to Cha\l upnik \cite{Chaluntwist} and Touz\'e \cite{Touzeuntwisting}.
Polynomial functors can also be used to compute cohomology for other classical algebraic groups (see \cite{TouzeAIM} and references therein). 
In a different direction, Hong, Touz\'e and Yacobi (\cite{HTY}, \cite{HY2}) showed that %if one considers $k$ with $\operatorname{char}(k) =p$, then 
there is a categorical action of $\widehat{\mathfrak{sl}_p}$ on $\mathcal{P}$ that categorifies the action of $\widehat{\mathfrak{sl}_p}$ on the Fock space, where $p$ is the characteristic of the base field.

Two instrumental properties in the theory of polynomial functors are representability and composition. Representability allows one to prove that the category $\mathcal{P}^d$ is equivalent to the category of modules of the Schur algebra $S(n;d)$ for any $n \geq d$, or the category of polynomial $\GL_n$ representations of degree $d$ (\cite[\S 3]{FS}). %The former category can be interpreted as the category of polynomial $\GL_n$ representations of degree $d$ and 
%This interpretation is crucial to the application of polynomial functors in rational cohomology of algebraic groups.(I don't like how this last sentence sounds@@@). 
Composition is a natural property of functors, which is not so natural for modules over a Schur algebra, hence is a main advantage of studying polynomial functors. 
It is extremely useful in performing cohomology calculations.
Given $F,G\in \mathcal P$ and an exact sequence in $\mathcal P$ that represents an element of $\Ext^n_\mathcal{P}(F,G)$, precomposing (respectively, postcomposing, when it is exact) the exact sequence with another polynomial functor $H\in\mathcal{P}$ gives a class in $\Ext^n_\mathcal{P}(FH,GH)$ (resp., $\Ext^n_\mathcal{P}(HF,HG)$). The special case of precomposing by the Frobenius functor, where the assignment is injective, is particularly interesting in many contexts (see \cite{FFSS}). 
In general, this can relate Ext spaces in different degrees and provides, for example, 
\cite[Theorem 2.13]{FS} which is a key technique in main computations in \cite{FS}. 

A natural question is whether one can deform the polynomial functors into quantum polynomial functors.
A first such quantization is due to Hong and Yacobi \cite{HongYacobi}. They introduced a new category $\pd$ that is a $q$-deformation of the category $\mathcal{P}^d$ and showed that it enjoys many properties that $\mathcal{P}^d$ has. As in the classical case where the corresponding category is equivalent to the module category for the Schur algebra, $\pd$ is equivalent to the category of finite dimensional modules for the quantum Schur algebra $S_q(n;d)$ of Dipper and James \cite{DipperJames}. They present several applications to quantum invariant theory, including a quantum version of $(\GL_n, \GL_m)$-Howe duality. However, their category of quantum polynomial functors does not allow for composition of the functors. 

The underlying reason is that the action of $S_d$ in equation \eqref{eq:homsp} is replaced by a quantum action that depends on the extra structure given to the objects. In \cite{HongYacobi}, the domain category of $\pd$ consists of pairs $(V_n, R_n)$, where one can think of $V_n$ as the defining $n$-dimensional $U_q(\mathfrak{gl}_n)$-module or the defining $A_q(n,n)$-comodule (where $A_q(n,n)$ is the quantum coordinate ring of $n \times n$ matrices). The generators of the braid group $\mathcal{B}_d$ act on $V_n^{\otimes d}$ via the standard $R$-matrix $R_n$, the $R$-matrix of the defining $U_q(\mathfrak{gl}_n)$-representation $V_n$. When one applies a quantum polynomial $F$ to $V_n$, there is an $U_q(\mathfrak{gl}_n)$ structure on $F(V_n)$ which produces an $R$-matrix $R_{F(V_n)}$ associated to $F(V_n)$. The problem is then that $R_{F(V_n)}$ is not a standard $R$-matrix, so the pair $(F(V), R_{F(V)})$ is not in the domain category. 

We can also see this in the classical case. The domain of a polynomial functor consists of vector spaces. One can endow such vector spaces with $\GL_n$-module structures. If $V$ is a polynomial representation of degree $e$, then $F \in \pd$ maps $V$ to a polynomial representation $F(V)$ of degree $de$. So we can think of the domain in the classical case as containing finite dimensional polynomial representations of $\GL_n$. The reason why this extra structure of $V$ is not present in the definition of polynomial functors is because $S_d$ acts on $V^{\otimes d}$ in the same way regardless of the $\GL_n$-structure on $V$. In other words, one can view a single polynomial functor as a functor between degree $e$ modules and degree $de$ modules for all $e\in\mathbb N$ at the same time. This is not true for quantum polynomial functors. The braid group acts on $V^{\otimes d}$ via the $R$-matrix of $V$, and different $U_q(\mathfrak{gl}_n)$-modules have different associated $R$-matrices. Therefore, the action of $\mathcal{B}_d$ on $V^{\otimes d}$ depends on the module structure of $V$. If $V$ is the defining representation, the action of the braid group on $V^{\otimes d}$ factors through the Hecke algebra $\mathcal{H}_d$, but if $V$ is the $e$-th tensor power of the defining representation, or its subquotient,, this action factors through the $e$-Hecke algebra $\mathcal{H}_{d,e}$ (see Definition \ref{def:qeHecke}). 
The $e$-Hecke algebras are quantizations of the symmetric group and together play the role which the symmetric group plays in the classical case.

This extra structure should be taken into account in the quantum case. It is then natural to introduce the categories $\pde$ of quantum polynomial functors that map degree $e$ modules to degree $de$ modules. 
In this category, we are able to compose quantum polynomials by Theorem \ref{thm:composition}. %We then prove in Theorem \ref{thm:composition} that one can compose polynomial functors. 
More precisely, given $G \in \mathcal{P}^{d_2}_{q, e}$ and $F \in \mathcal{P}^{d_1}_{q,d_2 e}$, their composition  $F \circ G$ is a polynomial in $\mathcal{P}_{q,e}^{d_1 d_2}$.

We think of the categories $\pde$ as ``higher'' analogues of $\pd$. They satisfy many of the properties that $\pd$ satisfy. For example we show in Theorem \ref{thm:tensorproductisbraided} that the category $\mathcal P_{q,e}=\bigoplus_d\pde$ is a braided monoidal category. Another fundamental property we prove is that $\pde$ has a (finite) generator when $q$ is generic (Theorem \ref{thm:finitegeneration}). The root of unity case is significantly harder than the generic $q$ case, since the domain category of the functors in $\pde$ consists of degree $e$ polynomial representations of $U_q(\mathfrak{gl}_n)$, which is more complicated then the corresponding category for generic $q$. However, when $e=1$, the category of polynomial representations of degree $1$ is ``the same'' for $q$ a root of unity and for generic $q$; it consists only of direct sums of the defining representation. Therefore our finite generation result holds for any $q$ when $e=1$. %as we explain in Remark \ref{eqtoHY}. We also note that our proof for finite generation is different from the one in \cite{HongYacobi}. We prefer this method since for $e>1$, it allows one to clearly see why the requirement for
%The condition for $\Gamma^{d,W^e_{m,d}}_{q,e}$ to be a generator for $\pde$ is $m \geq de$, where 
%$W^e_{m,d} := \oplus_{i=1}^d (V_{m}^{\otimes e})^{(i)}$.

The generator in Theorem \ref{thm:finitegeneration} is defined in terms of a direct sum even for $e=1$ (when we would hope our category to reduce to the category studied by Hong and Yacobi). This is something that is needed in order to define composition in full generality as we explain in Section \ref{sec:finitegeneration}. We then consider another category of polynomial functors, whose definition involves restricting the domain; we denote the new category by $\mathcal{P}^{\circ, d}_{q,e}$. This category has a projective generator as shown in Theorem \ref{thm:finitegenerationindecomposable} and we explain that when $e=1$ we get back the main result of \cite{HongYacobi} in full generality in Remark \ref{eqtoHY}.

The existence of the projective generators allows us to conclude the equivalence between the categories $\mathcal{P}^{\circ, d}_{q,e}$ and $\pde$ and the category of finite dimensional modules of certain Schur algebra. Such Schur algebras are natural generalizations of the quantum Schur algebra $S_q(n;d)$; we conjecture they also appear in a generalization of quantum Schur-Weyl duality. 

Another interesting difference between the quantum and the classical categories can be seen when taking into consideration $e$-Hecke pairs for all $e$ at once. 
In Section \ref{sec:where}, we define the category $\widetilde{\pd}$ as the category of functors with domain all $e$-Hecke pairs for all $e$. 
For $q=1$, this category is equivalent to the category $\pde$ for any $e$. 
When $q$ is generic, $\widetilde{\pd}$ is not equivalent to $\pde$ for any $e$. 

To further endorse the categories $\pde$ as quantum analogues of strict polynomial functors we give several examples of objects in $\pde$.  The most interesting we believe are the quantum symmetric and exterior powers of Berenstein and Zwicknagl \cite{BZ}. 

We now briefly outline the structure of the paper. 
In Section \ref{sec:prelim}, we introduce the basics of quantum multilinear algebra which are of use throughout the paper. In Section \ref{sec:def}, we define the categories $\pde$, the main objects to be studied in this paper and present several interesting examples of quantum polynomial functors. We focus on the quantum divided, symmetric and exterior power (due to Berenstein and Zwicknagl \cite{BZ}). 
We show that the definition of these objects produce quantum polynomial functors, so their construction fits in our framework. 
In Section \ref{section:braiding}, we show that the category of quantum polynomial functors is a braided monoidal category. In Section \ref{section:composition}, we explain how two quantum polynomial functors can be composed in our setting.  
In Section \ref{sec:finitegeneration}, we show that the categories $\pde$ and $\mathcal{P}^{\circ, d}_{q,e}$ have a (finite) projective generator for generic $q$. This immediately implies equivalence to the category of modules of a ``generalized'' $q$-Schur algebra. 
In Section \ref{sec:where}, we consider a different category $\widetilde{\pd}$ with domain all $e$-Hecke pairs for all $e \geq 0$ and we show that this category does not contain a projective generator for generic $q$. In Section \ref{sec:rootsofunity}, we discuss quantum polynomial functors at roots of unity.

%Polynomial functors have been proven useful in computations involving rational cohomology of algebraic groups. We hope that developing a theory of quantum polynomial functors will have interesting applications to the study of rational cohomology of quantum groups at roots of unity. (@@@ write some more here. HK wants to erase this paragraph  !!! V wants to include an upgraded version of this in final version). 
%Talk about Touze and quantum pol functors in other types(@@@ not sure if we should do it in this paper, maybe just start thinking about it and do it, seems like a  different project, since we need to do HY type BCD ). 

$\bf{Acknowledgements:}$ We would like to thank Catharina Stroppel, Antoine Touz\'e and Daniel Tubbenhauer for many helpful discussions. We would like to thank Jiuzu Hong and Oded Yacobi for offering helpful suggestions to a preliminary version of the paper. We would like to thank the Max Planck Institute for Mathematics in Bonn for excellent working conditions; the first author was supported by the Institute during the time the paper was written. We would like to thank an anonymous referee for a careful reading of the manuscript and helpful suggestions. 

\section{Preliminaries}\label{sec:prelim}

Let $k$ be a field. Let $q$ be an element of $k^\times$. We say $q$ is \textit{generic} if $q$ is not a root of unity. In this section we introduce several objects and prove some properties which will be of use in defining quantum polynomial functors. We note that several of these definitions and some of the properties are taken directly from \cite{HongYacobi}. 

\subsection{Yang-Baxter spaces}
Let $\mathcal{V}$ be the category of finite dimensional vector spaces over the
field $k$. Each $V \in \mathcal{V}$ comes with a chosen basis $\{ v_1, \cdots, v_n\}$ where $n =\text{dim}(V)$. Even though our results are independent of the chosen basis, the exposition is more clear if we associate a fixed basis to each vector space. 

Let $\tau:V \otimes W \to W \otimes V$ be the flip
operator, namely $\tau(v \otimes w) = w \otimes v$. Let $S_d$ be
the symmetric group on $d$ letters. Let $\mathcal{B}_d$ be the Artin braid
group generated by $T_i, \,1  \leq i \leq d-1$ subject to the relations
\begin{equation} \label{braidrelations}
\begin{split}
T_i T_j = T_j T_i \,\,\, \text{if} \,\,\, |i-j|>1 \\
T_i T_{i+1} T_i = T_{i+1} T_i T_{i+1}
\end{split}
\end{equation}
The Hecke algebra $\cH_d$ is the quotient of the braid group $\mathcal{B}_d$ by the relations 
\[(T_i-q)(T_i+q^{-1}) = 0, \forall i.\] 
For $V \in \mathcal{V}$, $R \in \text{End} (V \otimes V)$ is called an
$R$-$matrix$ if it satisfies the Yang-Baxter equation:
\begin{equation} \label{YBE}
R_{12}R_{23}R_{12} = R_{23}R_{12}R_{23}
\end{equation}
where $R_{12} = R \otimes 1_V \in \text{End}(V^{\otimes 3})$ and
$R_{23} = 1_V \otimes R \in \text{End}(V^{\otimes 3})$.

If $R \in \text{End}(V \otimes V)$ is an $R$-matrix, we call the pair $(V,R)$
a $\it{Yang}$-$\it{Baxter}$ $\it{space}$. To each pair we can associate a right representation,
$\rho_{d,V}:\mathcal{B}_d \to \text{End}(V^{\otimes d})$ that sends $T_i$ to $1_{V^{\otimes i}} \otimes R \otimes 1_{V^{\otimes d-i-1}}$. We will
most of the time use the short hand notation $V$ for the Yang-Baxter space $(V,R)$ and denote
the $R$-matrix in the pair $(V, R)$ by $R:=R_V$. 

We now define the quantum Hom-space algebra as it is defined by Hong and Yacobi
\cite{HongYacobi}. Given two Yang-Baxter spaces $V$ and $W$ with basis $\{ v_i \}$ and $\{ w_j \}$, respectively, let $T(V,W)$
be the tensor algebra of $\text{Hom}(V,W)$, that is
\[ T(V,W) = \oplus_{d \geq 0} T(V,W)_d \]
where $T(V,W)_d := \text{Hom}(V,W)^{\otimes d} \cong \text{Hom}
(V^{\otimes d}, W^{\otimes d})$. Let $I(V,W)$ be the two sided ideal
generated by
$ X \circ R_V - R_W \circ X,\textmd{   for all  } X \in \text{Hom}(V^{\otimes 2},
W^{\otimes 2}).$
Define $A(V,W) := T(V,W)/I(V,W)$. The space $A(V,W)$ has a natural gradation
\[A(V,W) = \oplus_{d \geq 0} A(V,W)_d\]
where $A(V,W)_d = T(V,W)_d/I(V,W)$.

Denote by
$x_{ji} : W \to V$ the map
\[ x_{ji} (w_k) = \delta_{kj} v_i \]
with $x_{ji} \in \text{Hom}(W,V) \subset A (W,V)$.

\begin{lem} \label{definition:AVW}
The algebra $A(W,V)$ has a presentation by the generators $x_{ji}$
and the relations generated by% the elements of the form
\begin{equation} \label{RTT}
 \sum_{k,l} (R^{pq}_{W,kl} x_{ki}x_{lj} - R^{kl}_{V,ij}x_{pk}x_{ql} ),
\end{equation}
where the coefficients $R^{kl}_{V,ij}$ are defined by the following equation:
\[ R_V (v_i \otimes v_j) = \sum_{k,l} R^{kl}_{V, ij} v_k \otimes v_l. \]
\end{lem}
\begin{proof}
Elements of the form $x_{ij}x_{kl}$ form a basis of $\text{Hom}(W^{\otimes 2}, V^{\otimes 2})$.
The quadratic relations \eqref{RTT} are exactly the relations that generate $R(W,V)$. Since $R(W,V)$ generates $I(W,V)$, the result follows.
\end{proof}

There is a degree preserving morphism of algebras
\[ \Delta_{V,W,U}: A(V,U) \to A(V,W) \otimes A(W,U) \]
that is given on generators by $\Delta_{V,W,U} (x_{ij}) =
\sum_k x_{ik} \otimes x_{kj}$. There is a map $V \to W \otimes \text{Hom}(W,V)$ given by
\[ v_i \mapsto \sum_j w_j \otimes x_{ji}. \]
This extends to a map
$\Delta_{V,W}:V \to W \otimes A(W,V)$.

\begin{prop} \label{diagram}
The following diagram commutes:
\begin{center}
\begin{tikzpicture}
  \matrix (m) [matrix of math nodes,row sep=3em,column sep=5em,minimum width=2em]
  {
     V & W \otimes A(W,V) \\
     U \otimes A(U,V) &U \otimes A(U,W) \otimes A(W,V) \\};
  \path[-stealth]
    (m-1-1) edge node [left] {$\Delta_{V,U}$} (m-2-1)
            edge node [above] {$\Delta_{V,W}$} (m-1-2)
    (m-2-1.east|-m-2-2) edge node [below] {$1 \otimes \Delta_{U,W,V}$} (m-2-2)
    (m-1-2) edge node [right] {$\Delta_{W,U} \otimes 1$} (m-2-2) ;
\end{tikzpicture}
\end{center}
\end{prop}

\begin{proof}
One can compute $(1 \otimes \Delta_{U,W,V})\Delta_{V,U} (v_i)= \sum_{k,j} u_j \otimes x_{jk} \otimes x_{ki} 
= (\Delta_{W,U} \otimes 1)\Delta_{V,W}(v_i)$
from which the commutativity of the diagram follows. 
\end{proof}

Let $(V, R_V)$ and $(W, R_W)$ be Yang-Baxter spaces. We define the generalized ($q$-)Schur algebra
\[ S(V,W;d) := (A(W,V)_d)^* \]
 as in \cite{HongYacobi}.
The following is proved in \cite{HongYacobi}:
\begin{prop} \label{prop:homspaceschuralgebra}
Let $V, W$ be Yang-Baxter spaces.
Then there is a natural isomorphism
\[ S(V,W;d) \cong \textnormal{Hom} _{\mathcal{B}_d}(V^{\otimes d},
W^{\otimes d}) \]
\end{prop}
\begin{proof}
See Proposition 2.7 in \cite{HongYacobi}.
\end{proof}

By taking the dual of $\Delta_{V,W,U}$ we obtain a map 
\[m_{U,W,V}: S(W,V;d) \otimes S(U,W;d) \to S(U,V;d).\]
There is a natural map $m_{U,W,V}' : \textnormal{Hom}_{\mathcal{B}_d}(W^{\otimes d}, V^{\otimes d}) \otimes \textnormal{Hom}_{\mathcal{B}_d}(U^{\otimes d}, W^{\otimes d})  \to \textnormal{Hom}_{\mathcal{B}_d}(U^{\otimes d}, V^{\otimes d})$ that takes $f \otimes g \mapsto f \circ g$. The following Proposition shows they are the same map under the isomorphism in Proposition \ref{prop:homspaceschuralgebra}.

\begin{prop}\label{HY2.8}
Given three Yang-Baxter spaces $V, U, W$, the following diagram commutes:

\begin{center}
\begin{equation*}
\begin{tikzpicture} 
  \matrix (m) [matrix of math nodes,row sep=3em,column sep=5em,minimum width=2em]
  {
     S(W,V;d) \otimes S(U,W;d) & \textnormal{Hom}_{\mathcal{B}_d}(W^{\otimes d}, V^{\otimes d}) \otimes \textnormal{Hom}_{\mathcal{B}_d}(U^{\otimes d}, W^{\otimes d})  \\
     S(U,V;d) & \textnormal{Hom}_{\mathcal{B}_d}(U^{\otimes d}, V^{\otimes d}) \\};
  \path[-stealth]
    (m-1-1) edge node [left] {$m_{U,W,V}$} (m-2-1)
            edge node [above] {$\cong$} (m-1-2)
    (m-2-1.east|-m-2-2) edge node [below] {$\cong$} (m-2-2)
    (m-1-2) edge node [right] {$m_{U,W,V}'$} (m-2-2) ;
\end{tikzpicture}
\end{equation*}
\end{center}
\end{prop}
\begin{proof}
See Proposition 2.8 in \cite{HongYacobi}.  
\end{proof}

\begin{rem}
If $W=V$, the quadratic relation \eqref{RTT} becomes the RTT relation due to
Faddeev, Reshetikhin and Taktajan. The algebra $A(V,V)$ is then
just the algebra denoted by $A_{R_V}$ in \cite{FRT}.
\end{rem}
We record two properties of $A(V,V)$ which are standard results in the theory of quantum matrices; their proofs are nothing more than simple computations.

\begin{enumerate}
\item $A(V,V)$ is a bialgebra with comultiplication $\Delta_{V,V,V} (x_{ij}) = \sum_k
x_{ik} \otimes x_{kj}$ and counit $ e(x_{ij}) = e_{A(V,V)}(x_{ij}) = \delta_{ij}$.
\item $V$ is an $A(V,V)$-comodule with coaction given by $\Delta_{V,V}(v_i) = \sum_j v_j \otimes x_{ji}$.
\end{enumerate}

We note that one of the two diagrams that need to commute for $\Delta_{V,V}$
to be a coaction is the diagram in Proposition \ref{diagram} for $V =W=U$.
Therefore we  can think of the map $\Delta_{V,W}: V \to W \otimes A(W,V)$ as
a generalization of the coaction.

Since the comultiplication is degree preserving (i.e., $\Delta_{V,V,V}$ maps $A(V,V)_d$ to
$A(V,V)_d \otimes A(V,V)_d$), the map $m_{V,V,V}$ makes $S(V,V;d)$ into an
algebra. The associativity of $m_{V,V,V}$ is equivalent to the coassociativity of
$\Delta_{V,V,V}$. The unit of $S(V,V;d)$ is given by the counit $e$ of $A(V,V)$.

\subsection{Quantum matrix spaces and $e$-Hecke pairs}

Let $V_n$ denote an $n$-dimensional vector space. Let
$R_n$ be the $R$-matrix of $U_q(\mathfrak{gl}_n)$ for the
defining representation, namely
\begin{equation}\label{def:Rmatrix} R_n (v_i \otimes v_j) = \begin{cases}
      v_j \otimes v_i  & \text{if} \,\, i<j \\
      qv_i \otimes v_j & \text{if} \,\, i=j \\
    (q-q^{-1})v_i \otimes v_j + v_j \otimes v_i & \text{if} \,\, i>j.
   \end{cases} 
 \end{equation}
It is well known that $(V_n, R_n)$ form a Yang-Baxter space. Define
$A_q(n,n):=A (V_n, V_n)$.

The space $A_q(n,n)$ is the algebra of quantum $n \times n$ matrices (see
\cite{Takeuchi} and \cite{FRT}). It is a coquasitriangular bialgebra
generated by elements $x_{ij}, 1 \leq i,j \leq n$ subject to the following
RTT relations:
\begin{equation} \label{RTT2}
\sum_{k,l} (R_n)^{pq}_{kl} x_{ki} x_{lj} = \sum_{k,l} (R_n)^{kl}_{ij}
x_{pk}x_{ql}. 
\end{equation}

We now present some standard properties of $A_q(n,n)$, see for example \cite[Chapter 7]{Lambe}. We begin by reminding the reader about the coalgebra structure on $A_q(n,n)$. The comultiplication and counit are given on generators by
\[ \Delta(x_{ij}) = \sum_k x_{ik} \otimes x_{kj}, \,\,\, \epsilon (x_{ij}) =
\delta_{ij}. \]
The vector space $V_n$ is an $A_q(n,n)$-comodule via the coaction $v_i \mapsto \sum_j v_j
\otimes x_{ji}$.

The coalgebra structure on $A_q(n,n)$ allows one to endow the tensor product $V \otimes W$ of two $A_q(n,n)$-comodules with the structure of an $A_q(n,n)$-comodule. Therefore the category of finite dimensional $A_q(n,n)$-comodules is a monoidal category. The unit is the trivial comodule $k$ with the coaction $1 \in k \mapsto 1 \otimes 1 \in k \otimes A_q(n,n)$. There are standard isomorphisms $l_V: k \otimes V \to V$ and $r_V: V \otimes k \to V$.

The bialgebra $A_q(n,n)$ is coquasitriangular. This means that there is a map $\mathcal{R}:
A_q(n,n) \otimes A_q(n,n) \to k$ that is invertible in the
convolution algebra, satisfying the following conditions:
\begin{equation} \label{equation:cqt}
\begin{split}
&\mathcal{R}(a b, c) = \mathcal{R}(a, c_{(1)}) \mathcal{R}(b, c_{(2)}) \\
&\mathcal{R}(a, bc) = \mathcal{R}(a_{(1)}, b) \mathcal{R}(a_{(2)}, c) \\
&b_{(1)} a_{(1)}\mathcal{R}(a_{(2)},b_{(2)}) = \mathcal{R}(a_{(1)}, b_{(1)})
b_{(2)} a_{(2)}
\end{split}
\end{equation}
for all $a,b,c \in A_q(n,n)$. In the above formula we use Sweedler notation, namely we denote $\Delta(a) = a_{(1)}\otimes a_{(2)}$. The map $\mathcal{R}$ is given on generators $x_{ij}$
by the formula
\begin{equation}\label{eq:universalR}
 \mathcal{R} (x_{ij} \otimes x_{kl}) = (R_n)^{ik}_{lj}.
 \end{equation}
The values of $\mathcal{R}$ on higher order terms is given by repeated
applications of the first two equalities in equation \eqref{equation:cqt}. 

The existence of $\mathcal{R}$ implies that for every
$A_q(n,n)$-comodules $V, W$ there is an $A_q(n,n)$-comodule isomorphism $R_{V,W}: V \otimes W \to
W \otimes V$ given by the formula
\begin{equation}\label{eq:Rmatrixcomodule} R_{V,W} := (1 \otimes 1 \otimes \mathcal{R})(1 \otimes \tau \otimes 1)(\Delta_{W} \otimes \Delta_{V})\tau.
\end{equation}
This morphism makes the category of finite dimensional $A_q(n,n)$-comodules into a strict braided monoidal category. A strict braided monoidal category $\mathcal{C}$ is a monoidal category with braiding isomorphisms $\gamma_{V, W} : V \otimes W \to W \otimes V$ that satisfy 
\begin{equation}\label{eq:braidedmonoidalcategory}
\begin{split}
&\gamma_{V \otimes  W, U} =  (\gamma_{V,U}\otimes 1) (1 \otimes \gamma_{W, U}) \\
&\gamma_{V, W \otimes U} = (1 \otimes \gamma_{V, U}) (\gamma_{V,W}\otimes 1) \\
&\tilde{r}_V \gamma_{I, V} = \tilde{l}_V, \,\,\, \tilde{r}_V \gamma_{V, I} = \tilde{l}_V 
\end{split}
\end{equation}
where $I$ is the identity object in the monoidal category and $\tilde{r}_V, \tilde{l}_V$ are the identity constraints in $\mathcal{C}$. 

\begin{prop}\label{prop:Aqnnbraidedmonoidalcategory}
The category of finite dimensional $A_q(n,n)$-comodules is a braided monoidal category with braiding isomorphisms given by $\gamma_{V,W} = R_{V,W}$.
\end{prop}

When $W= V$, the map $R_{V}:=R_{V,V}$ satisfies the Yang-Baxter equation. This makes makes the pair $(V,R_V)$ a Yang-Baxter space for every $A_q(n,n)$-comodule $V$. 

\begin{defn} \label{eHeckepair}
An $e$-$\it{Hecke}$ $\it{pair}$ $V$ is an $A_q(n,n)$-comodule for some $n\geq 1$ such that the image of the coaction $\Delta_V : V \to V \otimes A_q(n,n)$ lies in $V \otimes A_q(n,n)_e$.
\end{defn}

\begin{rem}\label{lemma:eHeckesubquotient}
Equivalently, an $e$-Hecke pair is a (finite dimensional) module over the $q$-Schur algebra $S_q(n;e)$, or equivalently a degree $e$ representation of $U_q(\mathfrak{gl}_n)$.
Yet another way to understand the $e$-Hecke pairs is to note the fact that they are direct sums of subquotients of $V_n^{\otimes e}$ where $V_n$ is either the $n$-dimensional defining comodule for $A_q(n,n)$ or the defining $U_q(\mathfrak{gl}_n)$ module.
\end{rem}

We explain the term ``$e$-Hecke pair''. Let $V$ be an $e$-Hecke pair. First, we call it a ``pair'' because we think of $V$ as the pair $(V, R_V)$. 
To explain the word ``Hecke'', let us start with the case $e=1$. 
If $V$ is indecomposable, then $V$ has to be the defining comodule $V_n$ 
and the action of the braid group $\mathcal{B}_d$ on $V_n^{\otimes d}$ factors through the action of the Hecke algebra $\cH_d$.  This is not the case for general $e$. Instead, the action of $\mathcal{B}_d$ on $V\td$ factors through a different deformation of the symmetric group $S_d$, which is realized as a subalgebra of $\cH_{de}$ as follows. Let $w_i$ be the element in $S_{de}$ such that
\begin{equation}\label{w} 
	w_i(j)=
	\begin{cases} 
      j+e   & \ \ \ e(i-1)<j\leq ei \\
      j-e & \hfill ei<j\leq e(i+1) \\
      j		& \hfill \text{otherwise}.
	\end{cases}
	\end{equation}
\begin{defn}\label{def:qeHecke}
The \textit{$e$-Hecke algebra of rank $d$}, denoted by $\cH_{d,e}$, is the subalgebra of $\cH_{de}$ generated by $T_{w_1},\cdots T_{w_{d-1}}$, where $w_i$ are as in equation \eqref{w}.
\end{defn} 

One sees from the definition that the action of $\mathcal{B}_d$ on $(V_n\te)\td$ factors through $\cH_{d,e}$.
Since an indecomposable $e$-Hecke pair $V$ is a subquotient of $V_n^{\otimes e}$ as an $A_q(n,n)$-comodule (see Remark \ref{lemma:eHeckesubquotient}), the $\mathcal B_d$-module $V\td$ is a $\mathcal B_d$-subquotient of $(V_n\te)\td$, hence a $\cH_{d,e}$-subquotient.
We also have that 
\begin{equation*}\label{eheckealgebra}
\Hom_{\cH_{d,e}}(V\td, W\td)=\Hom_{\mathcal{B}_d}(V\td, W\td)\cong(\Hom(V,W)\td)^{\mathcal{B}_d}.
\end{equation*}
It follows that the $d$-th tensor power of an indecomposable $e$-Hecke pair is a module over the $e$-Hecke algebra $\cH_{d,e}$. Note that the last sentence is false if we don't require the $e$-Hecke pair to be indecomposable, for both $e=1$ and general $e$. 

\begin{rem}
%The $e$-Hecke algebras $\mathcal{H}_{d,e}$ are ``higher'' deformations of the symmetric group $S_d$ that behave differently. 
Note that the dimension of $\mathcal{H}_{d,e}$ is in general greater than that of $kS_d$. 
%For example when $d=2$, the dimension of $\mathcal H_{2,e}$ is larger than 2 unless $e=1$. In order to compute the dimension of $\mathcal H_{2,e}$, we can relate it to the number of eigenvalues of a certain $R$-matrix. 
When $d=2$, we can relate the dimension of $\mathcal H_{d,e}$ to the eigenvalues of a certain $R$-matrix, where we can already see the difficulty of computing the dimension.
%The algebra $\mathcal{H}_{2,e}$ has one generator $T_w$. 
The Schur-Weyl duality identifies $\mathcal{H}_{2e}$ with $\operatorname{End}_{U_q(\mathfrak{gl}_{2e})}(V_{2e}^{\otimes 2e})$; under this identification, the generating element $T_w\in\mathcal H_{2,e}$ corresponds to the map $R_{V_{2e}^{\otimes e}}$. Thus, the dimension of $\mathcal{H}_{2,e}$ is equal to the degree of the minimal polynomial of $R_{V_{2e}^{\otimes e}}$. %is equal to the dimension of % (since $T_w$ is the only generator)
Since any $R$-matrix is diagonalizable for $q$ generic, the degree %This degree is equal to the degree of the minimal polynomial of $R_{V_{2e}^{\otimes e}}$, which 
is equal to the number of different eigenvalues of the $R$-matrix $R_{V_{2e}^{\otimes e}}$%(since the $R$-matrix is diagonalizable when $q$ is generic)
. For example, when $e=2$ and $q \neq 1$ is generic, the $R$-matrix $R_{V_4^{\otimes 2}}$ has $7$ distinct eigenvalues as seen from the table in the last Section of \cite{HongYacobi}. Therefore, the dimension of $\mathcal{H}_{2,2}$ is $7$. 
For general $e$, we do not know the number of eigenvalues for the $R$-matrices involved. But the argument above gives an upper bound $2(2e^2+1)$ for the dimension of $\mathcal{H}_{2,e}$. %, showing that $\mathcal{H}_{2,e}$ is not all of $\mathcal H_{2e}$.
\end{rem}

\begin{rem}
The $e$-Hecke algebra arises when one considers the wreath product of finite groups. Given two groups $G$ and $H\subseteq \mathcal{B}_d$, one can define a wreath product mimicking the usual construction where $H\subseteq S_d$: Let the wreath product $G\wr H$ be the semidirect product $G^{\times d} \rtimes H$, where the action of $H$ is the braid group action permuting components.
If $G=\mathcal{B}_e$ and $H=\mathcal{B}_d$, then the wreath product $\mathcal{B}_e\wr \mathcal{B}_d$ is a subgroup of the larger braid group $\mathcal{B}_{de}$ generated by $T_1,\cdots,T_{e-1},T_{e+1},\cdots,T_{2e-1},\cdots, T_{(d-1)e+1},\cdots,T_{de-1}$ and $w_1,\cdots,w_{d-1}$, where $T_i$ are the standard generators for $\mathcal B_{de}$ and $w_i\in\mathcal B_{de}$ are (unique) shortest lifts of $w_i\in S_{de}$ in \eqref{w}.
Now we replace the braid groups by Hecke algebras in constructing the (internal) wreath product. The group $\mathcal{B}_e$ is replaced by the Hecke algebra $\cH_e$, and the product $\mathcal{B}_e^{\times d}$ is replaced by $\cH_e\td$. This latter algebra is a subalgebra of $\cH_{de}$ generated by $T_1,\cdots, T_{e-1},T_{e+1},\cdots,T_{2e-1},\cdots,T_{(d-1)e+1},\cdots,T_{de-1}$. The $e$-Hecke algebra $\cH_{d,e}$ naturally acts on this; the generator $T_{w_i}$ acts as the multiplication in $\cH_{de}$. Then, the subalgebra in $\mathcal{H}_{de}$ generated by the above $\cH_e\td$ and our $e$-Hecke algebra $\cH_{d,e}$, denoted by $\mathcal H_e\wr \mathcal H_{d,e}$, can be thought of as an analog to $\mathcal{B}_e\wr \mathcal{B}_d$ or $S_e\wr S_d$.

%The essential step in composing the quantum polynomials can now be written as $\Hom_{\cH_{de}}(V^{\otimes de},W^{\otimes de})\subset \Hom_{\cH_e\td\otimes\cH_d^e}(V^{\otimes de}, W^{\otimes de})$.
\end{rem}

The last equality of equation \eqref{equation:cqt} implies that $R_V$ is an
$A_q(n,n)$-comodule homomorphism.
Given an $A_q(n,n)$-comodule $V$, write the coaction map as 
\[ \Delta_V: v_i \mapsto \sum_j v_j \otimes t^V_{ji} \]
for some $t^V_{ji} \in  A_q(n,n)$. The equation above serves as the definition for $t^V_{ji}$. 

\begin{lem} \label{RWcomodulehomomorphismequivalence}
The equation 
\[ \sum_{k,l} R^{pq}_{V,kl} t^V_{ki} t^V_{lj} = \sum_{k,l} R^{kl}_{V,ij} t^V_{pk}t^V_{ql}. \]
is equivalent to the fact that $R_V$ is an $A_q(n,n)$-comodule homomorphism.
\end{lem}
\begin{proof}
$R_V$ is an $A_q(n,n)$-comodule homomorphism if and only if 
\[ \Delta_{V \otimes V} R_V = (R_V  \otimes 1) \Delta_{V \otimes V}. \]
Applying both sides to $v_i \otimes v_j$ and picking out the $A_q(n,n)$-coefficients of $v_p \otimes v_q$ produces the 
the desired equation. 

\end{proof}

\begin{lem} \label{RTTwelldefined}
The relation 
\begin{equation}\label{eqRTTrel}
\sum_{k,l} R^{pq}_{V,kl} t^V_{ki} t^V_{lj} = \sum_{k,l} R^{kl}_{V,ij}
t^V_{pk}t^V_{ql}. 
\end{equation}
holds in $A_q(n,n)$. 
\end{lem}
\begin{proof}
This follows from Lemma $\ref{RWcomodulehomomorphismequivalence}$ and the fact that 
$R_V$ is an $A_q(n,n)$-comodule homomorphism. 
\end{proof}

Given an $A_q(n,n)$-comodule $V$, we denote $A(V,V)$ by $A_q(V,V)$ and $S(V,V;d)$ by $S_q(V,V;d)$.
If we define $S_q(n;d) := S_q(V_n,V_n;d)$, then $S_q(n;d)$ is the $q$-Schur algebra due to
Dipper and James \cite{DipperJames}.

\begin{prop} \label{semisimple}
Suppose $q$ is generic or $1$ and $\operatorname{char} k = 0$. Then the category of finite dimensional $A_q(n,n)$-comodules is semisimple.
\end{prop}
\begin{proof}
See Theorem 11.4.4 in Parshall and Wang \cite{PW} where our Proposition is proved for $\GL_q(n)$. The same result for $A_q(n,n)$ follows in a similar way.

%The category of finite dimensional $\GL_q(n)$-comodules is semisimple (see \cite{Hai} chapter 3).
%Suppose $W$ is an $A_q(n,n)$-subcomodule of $V$, that is, we have a commutative diagram  
%\[\begin{CD}
%W@>{\Delta_W}>>W\otimes A_q(n,n)\\
%@V{f}VV @V{f\otimes 1}VV\\
%V@>{\Delta_V}>> V \otimes A_q(n,n)
%\end{CD}\]
%where $f$ is an inclusion and $\Delta_W, \Delta_V$ are the coactions. 
%We have a coalgebra inclusion $\iota: A_q(n,n)\to \GL_q(n)$, via which $W,V$ are $\GL_q(n)$-comodules. The outer maps $f, f\otimes 1_{\GL_q(n)}$ in the diagram
%$$\begin{CD}
%W@>{\Delta_W}>>W\otimes A_q(n,n)@>{1\otimes \iota}>>W\otimes \GL_q(n)\\
%@V{f}VV @V{f\otimes 1}VV @V{f\otimes 1}VV\\
%V@>{\Delta_V}>> V \otimes A_q(n,n)@>{1\otimes \iota}>>V\otimes \GL_q(n)
%\end{CD}$$
%have sections. Restricting it to the left diagram, we have that $W$ is a direct summand of $V$ as an $A_q(n,n)$-comodule. This proves the claim.
\end{proof}

\begin{prop} \label{prop:subcomodule}
Suppose $q$ is generic or $1$ and $\operatorname{char} k = 0$. Any irreducible $A_q(n,n)_e$-comodule $V$ is a direct summand of $V_n^{\otimes e}$.
\end{prop}
\begin{proof}
This follows from Proposition \ref{semisimple} and Remark \ref{lemma:eHeckesubquotient}.
\end{proof}

\section{Quantum polynomial functors}\label{sec:def}

\subsection{Definition}\label{subsection:definition}

In this section we propose a different definition of quantum polynomial functors that generalizes the definition in \cite{HongYacobi}. Our category of quantum polynomial functors enjoys many of the properties presented in \cite{HongYacobi}, and additionally, it has a composition. 
Let $d,e$ be non-negative integers.

Let us define the quantum divided power category
$\Gamma^d_{q,e} \mathcal{V}$. Its objects are all $e$-Hecke pairs for all positive $n$. 
The morphisms are defined as
 \[ \text{Hom}_{\Gamma^d_{q,e}} (V, W):= \text{Hom}_{\mathcal{B}_d}
(V^{\otimes d}, W^{\otimes d}). \]
Note that the category $\Gamma^d_{q,e} \mathcal{V}$ always contains a one dimensional $e$-Hecke pair that is obtained by tensoring the defining $A_q(1,1)_1$-comodule with itself $e$ times. It then follows that the category contains an $n$-dimensional vector space for every positive integer $n$.  When $d=1$, the forgetful functor $(V,R)\mapsto V$ to finite dimensional vector spaces induces an equivalence of categories 
\begin{equation}\label{1hecke}
\Gamma^1_{q,e}\mathcal V\cong \mathcal V
\end{equation}
for any $q, e$.
When $q=1$, the $R$-matrix $R_V$ of any $e$-Hecke pair $(V,R_V)$ is just the transpose map $V\otimes V\to V\otimes V,\ v\otimes w\mapsto w\otimes v$, thus we also have the equivalence
\begin{equation}\label{q=1}
\Gamma^d_{1,e}\mathcal V\cong \Gamma^d\mathcal V 
\end{equation} for any $d,e$, where $\Gamma^d\mathcal V$ is the domain category for the classical polynomial functors defined in the introduction.

\begin{defn}\label{def:qpf}
A quantum polynomial functor of degree $d$ on $e$-Hecke pairs is a linear functor
\[ F : \Gamma^d_{q,e} \mathcal{V} \to \mathcal{V}. \]
\end{defn}

We denote by $\mathcal{P}^d_{q,e}$ the category of quantum polynomial functors of degree $d$ on $e$-Hecke pairs. Morphisms are natural transformations of functors. 

\begin{defn}\label{def:qpfcirc}
Define the category $\mathcal{P}^{\circ, d}_{q,e}$ as in Definition \ref{def:qpf} with the added requirement that the domain consists only of $e$-Hecke pairs that are a subquotient of $V_n^{\otimes e}$ for some $n$. 
\end{defn}

Given an object of $\pde$, one can restrict its domain and define an object of $\mathcal{P}^{\circ, d}_{q,e}$. We will show that the categories $\pde$ are the natural setting where one can define composition. 
%We will use the short expression a functor of $type$ $(d,e)$ to denote such a functor. 
The equivalence \eqref{q=1} tells us that if we specialize $q=1$ in in both categories defined above, we recover the category of polynomial functors $\mathcal P^d$ of Friedlander-Suslin \cite{FS} (see Introduction).

\begin{rem}
If $d=0$, $\Hom_{\Gamma^{d}_{q,e}} (V^{\otimes d},W^{\otimes d}) = \Hom (k,k)$. Therefore the constant functor, mapping an $e$-Hecke pair $V \mapsto k$, where $k$ is the trivial $A_q(n,n)$-comodule, is a degree $0$ polynomial functor on $e$-Hecke pairs. It's not hard to see all elements in $\mathcal{P}^d_{q,e}$ are direct sums of the constant functor. 
\end{rem}

There is an equivalent characterization of a polynomial functor both in the classical and Hong and Yacobi \cite{HongYacobi} setting, which directly applies to ours. Given $F\in\pde$,
$V, W \in \Gamma^d_{q,e} \mathcal{V}$, we have a map
\[F_{V,W} : S_q(V,W;d) \to \text{Hom}(F(V),F(W))\]
which gives rise to two maps
\[ F'_{V,W}: S_q(V,W;d) \otimes F(V) \to F(W) \]
\[ F_{V,W}'': F(V) \to F(W) \otimes A_q(W,V)_d .\]

\begin{prop} \label{equivalentcharacterization}
(\cite{HongYacobi}, Proposition 3.5)
 A quantum polynomial functor $F$ of degree $d$ is equivalent
to the following data:
\begin{enumerate}
\item for each $V \in \Gamma^d_{q,e} $ a vector space $F(V)$;
\item given $V, W \in \Gamma^d_{q,e} $, a linear map
\[ F_{V,W}'': F(V) \to F(W) \otimes A_q(W,V)_d \]
such that the following diagrams commute
for any $V, W, U \in \Gamma^d_{q,e}$:

\begin{center}
\begin{equation}\label{polyfunctor1}
\begin{tikzpicture} 
  \matrix (m) [matrix of math nodes,row sep=3em,column sep=5em,minimum width=2em]
  {
     F(V) & F(U) \otimes A_q(U,V)_d \\
     F(W) \otimes A_q(W,V)_d & F(U) \otimes A_q(U,W)_d \otimes A_q(W,V)_d \\};
  \path[-stealth]
    (m-1-1) edge node [left] {$F_{V,W}''$} (m-2-1)
            edge node [above] {$F_{V,U}''$} (m-1-2)
    (m-2-1.east|-m-2-2) edge node [below] {$1 \otimes F_{W,U}''$} (m-2-2)
    (m-1-2) edge node [right] {$1 \otimes \Delta_{U,W,V}$} (m-2-2) ;
\end{tikzpicture}
\end{equation}
\end{center}

\begin{center}
\begin{equation} \label{polyfunctor2}
\begin{tikzpicture}
  \matrix (m) [matrix of math nodes,row sep=3em,column sep=5em,minimum width=2em]
  {
     F(V) & F(V) \otimes A_q(V,V)_d \\
     F(V)\otimes k  &  \\};
  \path[-stealth]
    (m-1-1) edge node [above] {$F_{V,V}''$} (m-1-2)
    (m-1-1) edge node [left] {$1$} (m-2-1)
    (m-1-2) edge node [below] {$\,\,\,\,\,\,\,\,\,\,1 \otimes e$} (m-2-1);
\end{tikzpicture}
\end{equation}
\end{center}
\end{enumerate}
\end{prop}

\begin{proof}
See Proposition 3.5 in \cite{HongYacobi}. We note that even
though they prove it for when $V, W, U$ are defining comodules, their proof goes through unchanged
for general $e$-Hecke pairs as in our setting. 
\end{proof}

We can extend the map $F''_{V,V}: F(V) \to F(V) \otimes A_q(V,V)_d$ to a map $\Delta^V_{F(V)}: F(V) \to F(V) \otimes A_q(V,V)$ that 
satisfies the following property:

\begin{lem} \label{comodule2}
The map $\Delta^V_{F(V)}$ makes $F(V)$ into an
$A_q(V,V)$-comodule.
\end{lem}
\begin{proof}
The diagrams in Proposition $\ref{equivalentcharacterization}$ are exactly what is
needed for $\Delta^V_{F(V)}$ to be a coaction map.
\end{proof}

Let $V$ be an $e$-Hecke pair. The bialgebra $A_q(V,V)$ is a quotient of the free algebra generated by $\{ x^V_{ji} \}$ 
by the ideal generated by $ \sum_{k,l} ( R^{pq}_{V,kl} x^V_{ki} x^V_{lj} - R^{kl}_{V,ij}x^V_{pk}x^V_{ql})$. 
The space $V$ is an $A_q(V,V)$-comodule via the coaction 
\[ \Delta^V_V: v_i \mapsto \sum_j v_j \otimes x^V_{ji}\]
and let
\[ \Delta_V: v_i \mapsto \sum_j v_j \otimes t^V_{ji} \]
be the coaction map that makes $V$ into an $A_q(n,n)$-comodule. 

Define the map $\psi_V : A_q(V,V) \to A_q(n,n)$ on the 
generators of $A_q(V,V)$ as follows: 
\begin{equation}\label{eq:psiV} 
\psi_V(x^V_{ji}) = t^V_{ji}.
\end{equation}

\begin{lem}\label{psiV}
The map $\psi_V$ is a bialgebra homomorphim. 
\end{lem}
\begin{proof}

We first need to check that $\psi_V$ is well defined. We can do this by showing that :  
\[ \psi_V ( \sum_{k,l} R^{pq}_{V,kl} x^V_{ki}x^V_{lj} - R^{kl}_{V,ij}x^V_{pk}x^V_{ql}
)  =   \sum_{k,l} R^{pq}_{V,kl} t^V_{ki}t^V_{lj} - R^{kl}_{V,ij}t^V_{pk}t^V_{ql}  = 0. \]
The first equality is by definition and the second one holds according to Lemma $\ref{RTTwelldefined}$. 

We now show that $\psi_V$ is a coalgebra homomorphism. This is equivalent to showing
that $\psi_V$ commutes with the comultiplication and the counit, namely 
\[ \Delta_{A_q(n,n)} \psi_V = (\psi_V \otimes \psi_V) \Delta_{A_q(V,V)}. \]
\[ e_{A_q(n,n)} \psi_V = e_{A_q(V,V)} \]
Both equations follows from the fact that $\Delta_V$ is a coaction. 
\end{proof}

Let $G$ be a quantum polynomial functor of degree $d$. We have maps:
\[ G_{V,W}'' : G(V) \to G(W) \otimes A_q(W,V)_{d}. \]

If we denote by $\{ v_i^G\}$ and $\{w^G_j\}$ the bases of $G(V)$ and $G(W)$, respectively, the map $G_{V,W}''$ takes $v^G_i \mapsto w^G_j \otimes t^{W,V}_{ji}$ for $t^{W,V}_{ji} \in A_q(W,V)_{d}$. 
By the definition of $A_q(G(W),G(V))$ we get a map 
\[G(V) \to G(W) \otimes A_q(G(W),G(V))\] 
which takes 
\[v^G_i \mapsto w^G_j \otimes x^{W,V}_{ji},\]
where $x^{W,V}_{ji}$ are the generators of $ A_q(G(W),G(V))$. 

\begin{lem}\label{psigeneral}
The map 
\begin{equation}\label{psi}
\psi^G_{W,V} : A_q(G(W),G(V)) \to A_q(W,V)
\end{equation}
defined on generators as $\psi^G_{W,V}(x^{W,V}_{ji}) :=  t^{W,V}_{ji}$ is well-defined. Therefore it is an algebra homomorphism. 
\end{lem}
\begin{proof}
The proof is similar to how we proved that $\psi_V$ is well-defined in Lemma \ref{psiV}. We want 
\[ \psi^G_{W,V}  ( \sum_{k,l} R^{pq}_{G(W),kl} x^{W,V}_{ki}x^{W,V}_{lj} - R^{kl}_{G(V),ij}x^{W,V}_{pk}x^{W,V}_{ql}
)  =   \sum_{k,l} R^{pq}_{G(W),kl} t^{W,V}_{ki}t^{W,V}_{lj} - R^{kl}_{G(V),ij}t^{W,V}_{pk}t^{W,V}_{ql}  = 0. \]
The first equality is by definition. The second equality follows from the commutativity of the exterior square in the diagram
\begin{equation}\label{eq:diagramGWV}
\begin{CD}
G(V) \otimes G(V)@>{(1\otimes m)(1 \otimes \tau \otimes 1)G_{V,W}''^{\otimes 2}}>>G(W) \otimes G(W) \otimes  A_q(G(W),G(V))@>{1\otimes \psi^G_{W,V}}>>G(W) \otimes G(W) \otimes  A_q(W,V)\\
@V{R_{G(V)}}VV @V{R_{G(W)}\otimes 1}VV @V{R_{G(W)}\otimes 1}VV\\
G(V) \otimes G(V)@>{(1\otimes m)(1 \otimes \tau \otimes 1)G_{V,W}''^{\otimes 2}}>> G(W) \otimes G(W) \otimes  A_q(G(W),G(V))@>{1\otimes \psi^G_{W,V}}>>G(W) \otimes G(W) \otimes  A_q(W,V)
\end{CD}
\end{equation}
where $m$ is the multiplication in $A_q(GW,GV)$.
The diagram is commutative because the two small squares are: the first by Lemma \ref{RWcomodulehomomorphismequivalence} and the second is trivial.

\end{proof}

\begin{lem}\label{psiG}
The maps $\psi^G_{W,V}$ satisfy the following two equations for all $e$-Hecke pairs $U,V,W$:
\begin{equation} \label{eq:psiG1}
\Delta_{U,W,V} \circ \psi^G_{U,V} = (\psi^G_{U,W} \otimes \psi^G_{W,V}) \circ \Delta_{G(U),G(W), G(V)}
\end{equation}
\begin{equation}\label{eq:psiG2}
e_{A_q(G(V),G(V))} = e_{A_q(V,V)} \psi^G_{V,V} 
\end{equation}
\end{lem}
\begin{proof}
The first equation is equivalent to $\Delta_{U,V,W}(t^{U,V}_{ij})=\sum_j t^{U,W}_{ik}\otimes t^{W,V}_{kj}$. This in turn follows from the commutativity of diagram (\ref{polyfunctor1}).

The second equation is equivalent to $e_{A_q(V,V)} (t^{W,V}_{ij}) = \delta_{ij}$ which follows from the commutativity of diagram (\ref{polyfunctor2}).
\end{proof}

\subsection{Basic operations on quantum polynomial functors}\label{subsection:basicproperties}

We denote by $\mathcal P^d_q:=\bigoplus_e \pde$ the category of quantum polynomial functors of degree $d$, by $\mathcal P_{q,e}:=\bigoplus_d\pde$ the category of quantum polynomial functors on $e$-Hecke pairs, and by $\mathcal P_q:=\bigoplus_{e,d}\pde$ the category of quantum polynomial functors. 

\subsubsection{Tensor product in $\mathcal P_{q,e}$}
Given two quantum polynomials on $e$-Hecke pairs $F\in\pde$ and $G\in\mathcal P^{d'}_{q,e}$, the (external) tensor product $F\otimes G \in\mathcal P^{d+d'}_{q,e}$ is defined in the same way as in \cite{HongYacobi}. For an $e$-Hecke pair $V$, we define it to be 
\[ (F\otimes G)(V)=F(V)\otimes G(V).\]
 To define it on the morphisms one uses the inclusion $\mathcal{B}_d \times \mathcal{B}_{d'}\subset \mathcal{B}_{d+d'}$. To be more explicit, for two $e$-Hecke pairs $V,W$, the map $(F\otimes G)_{V,W}$ is the following composition:
\begin{align*}
\Hom_{\mathcal{B}_{d+d'}}(V^{\otimes d+d'},W^{\otimes d+d'})&\inj\Hom_{\mathcal{B}_d\times \mathcal{B}_{d'}}(V^{\otimes d}\otimes V^{\otimes d'},W^{\otimes d}\otimes W^{\otimes d'})\\
&\to\Hom_{\mathcal{B}_d}(V^{\otimes d},W^{\otimes d})\otimes\Hom_{\mathcal{B}_{d'}}( V^{\otimes d'},W^{\otimes d'})\\
&\to\Hom(F(V),F(W))\otimes\Hom(G(V),G(W))\\
&\to\Hom(F(V)\otimes G(V),F(W)\otimes G(W))
\end{align*}
where the third map is $F_{V,W}\otimes G_{V,W}$. 

Recall the constant functor $k \in \mathcal{P}^0_{q,e}$. It maps an $e$-Hecke pair $V$ to the trivial $A_q(n,n)$-comodule $k$. It is then an easy exercise using the definition above to show that $F \otimes k \cong F$ via the natural transformation $\eta : F \otimes k \to F$ given by the standard isomorphism $\eta_V: F(V) \otimes k \cong F(V)$. We similarly have that $k \otimes F \cong F$. This can be summarized as follows: 

\begin{prop}
The category $\mathcal{P}_{q,e}$ is a monoidal category with the tensor product $\otimes$ and the unit object $k$. 
\end{prop}

\subsubsection{Duality in $\pde$}\label{duality}

One defines a duality on the functor category using dualities on the domain and codomain categories. 
In $\mathcal V$, we have the linear dual $V\mapsto V^*=\Hom_k(V,k)$. 
For our category $\Gamma^d_{q,e}\mathcal V$ where the objects have additional structures, we ``lift'' the linear dual to what is compatible with this structure, namely the twisted dual. 
To explain the twist here, it is more convenient to work with the  $U_q(\mathfrak{gl}_n)$-modules where the twisted duality is rather standard. (Recall that an $A_q(n,n)$-comodule can be thought of as a polynomial representation for $U_q(\mathfrak{gl}_n)$.) 
For a $\uqgln$-module $V$, one can define a twisted $\uqgln$ structure on the linear dual $V^*$ of the underlying vector space by precomposing the $\uqgln$ action by an antiautomorphism $\tau_1$ (see \cite[9.20]{quantumJ} for the definition of $\tau_1$ and the twisted dual). We denote by $^\tau V$ the resulting $\uqgln$-module. 
Then we have
\begin{equation}\label{doubledual}
^\tau(^\tau V)\cong V
\end{equation}
in $\uqgln$-mod.
We remark that $^\tau -$ is a duality of a highest weight category, under which the irreducibles are self-dual and a standard module and a costandard module of the same highest weight are dual to each other. 
In particular, the duality preserves the degree of polynomial representations, that is, if $V$ is an $\uqgln$-module of degree $e$, then the dual $^\tau V$ is also a degree $e$ $\uqgln$-module, so we can extend this duality on the category of $A_q(n,n)$-comodules. We therefore have a contravariant functor \[^\tau-:\Gamma^d_{q,e}\mathcal V\to\Gamma^d_{q,e}\mathcal V.\]
 %, however, the objects has an additional structure which is not preserved by the linear dual: the linear dual of an $e$-Hecke pair is 

Now we define the duality $-^\#$ on $\pde$ as
  \[F^\#:=^\tau \!\! F(^\tau- ).\]
It is not hard to show that $F^\#$ satisfies the properties in Proposition \ref{equivalentcharacterization} and therefore it is a quantum polynomial functor. Taking $q=1$ in our setting agrees with the classical definition of the dual.

\subsection{Examples}\label{examples}
We present several examples of quantum polynomial functors. 

\subsubsection{Tensor powers} For each $d\in \N$, the $d$-th tensor product functor $\otimes^d :V\mapsto V^{\otimes d}$ is a quantum polynomial of degree d. To be more precise, for each $e$, there is a functor $\otimes ^d_e$ in $\mathcal P^d_{q,e}$
which maps an $e$-Hecke pair $V$ to the $de$-Hecke pair $V^{\otimes d}$. The map $\otimes^d_{V,W}$ on the morphisms is just the inclusion
\[\Hom_{\mathcal{B}_d}(V,W)\inj\Hom(V^{\otimes d},W^{\otimes d}).\]
We abusively denote this functor by $\otimes^d$ for any $e$.

When $d=0$ we get the constant polynomial functor $\otimes ^0:V\mapsto k$, which we denote abusively by $k$,
and $\otimes ^1$ is the identity functor which we prefer to denote by $I$. %this can be noted after we do tensor product-Note that we have $\otimes^d\otimes\otimes^{d'}\cong \otimes^{d+d'}$.

\subsubsection{Divided powers and symmetric powers}\label{sub:dividedpowers}
Given an $e$-Hecke pair $W$, the divided power $\Gamma^{d,W}_{q,e}$ is the object in $\pde$ represented by $W$, namely,
$\Gamma^{d,W}_{q,e} :V\mapsto \Hom_{\mathcal{B}_d}(W^{\otimes d},V^{\otimes d})$. The map on morphisms  
\[\Hom_{\mathcal{B}_d}(V_1^{\otimes d},V_2^{\otimes d})\to \Hom(\Hom_{\mathcal{B}_d}(W^{\otimes d},V_1^{\otimes d}), \Hom_{\mathcal{B}_d}(W^{\otimes d},V_2^{\otimes d}))\] 
is given by $f\mapsto f\circ -$. 
When $W$ is the $e$-Hecke pair $(k^{\otimes e},q^{e^2})$, we denote it by $\Gamma^d_{q,e}$ and call it the divided power.
We may drop the index ``$e$'' in the notation because it is determined by $V$. That is, if $V$ is an $e$-Hecke pair, then $\Gamma^{d,V}_q$ denotes the functor $\Hom_{\Gamma^d_{q,e}\mathcal V}(V,-)$ in $\pde$.
Since a divided power is a representable functor, Yoneda's lemma tells us that it is a projective object in $\pde$ (see Proposition \ref{prop:projective} for a detailed proof).
Moreover, divided powers form a set of projective generators for $\pde$ for generic $q$ as we shall prove in Theorems \ref{thm:finitegeneration} and  \ref{thm:finitegenerationindecomposable}.
% More generally, $\Gamma^{d,V}_q = \Hom_{\mathcal{B}_d}(V^{\otimes d},-^{\otimes d})$ is the functor represented by $V$, hence a projective.
 
For an $e$-Hecke pair $V$ let $S^d_V\in\pde$ be defined by
\[S^d_V := (\Gamma^{d,V}_{q,e})^{\#}\]
where the dual $F^{\#}$ of a functor $F$ is defined in \S\ref{duality}.
These are the corepresentable objects in $\pde$ and form a set of injective cogenerators for $\pde$.

\subsubsection{Quantum symmetric, divided and exterior powers} %$S^d_{q,e}:V\mapsto V^{\otimes d} \otimes_{\mathcal{B}_d} k^{\otimes d}$, the q-symmetric power. As above, we can define $S^d_{q,V}$ as the functor $-\t_{\mathcal{B}_d}V\td$.
In this subsection we assume $\operatorname{char}(k)=0$. Let $V$ be an $e$-Hecke pair.
An important set of examples of quantum polynomial functors are the quantum symmetric and exterior powers for generic $q$ due to Berenstein and Zwicknagl \cite{BZ} (they also require $\operatorname{char}(k)=0$). These are quantum deformations of the classical symmetric and exterior power, though their theory is significantly more complex. For example, it is not necessary that the dimension of the quantum symmetric power of an $\uqgln$-module to be the same as the dimension of the classical symmetric power of the corresponding $U(\mathfrak{gl}_n)$-module. 

In this subsection we define the quantum symmetric, divided and exterior powers when $q$ is generic or when it is a $p$-th root of unity for an odd integer $p$. In the generic case, the definition of quantum symmetric and exterior can easily be seen to be the same as in \cite{BZ}. 

%We note that Berenstein and Zwicknagl work with $U_q(\mathfrak{sl}_n)$-modules $V$, and we work with $A_q(n,n)$-comodules $V$, but most of their results hold in our setting as well due to the fact that the braided monoidal category of finite dimensional $A_q(n,n)$-comodules embeds into the category of finite dimensional $U_q(\mathfrak{sl}_n)$-modules. 

Let $V$ be an $e$-Hecke pair. The map $R_{V}: V^{\otimes 2} \to V^{\otimes 2}$ is diagonalizable for generic $q$. For all non-zero $q$ it has eigenvalues $\pm q^r$, for a finite number of integers $r$ (see \cite{BZ} for details; note that they only state it for when $V$ is indecomposable, but it is easy to see that this holds for the direct sum of indecomposables as well). If $q$ is a $p$-th root of unity and $p$ is an odd integer, then $q^a \neq -q^b$ for any integers $a, b$. Therefore the following definition makes sense for this choice of $q$, even though the $R$-matrix $R_V$ might not be diagonalizable. 
Define 
\begin{equation}\label{eq:extanddiv}
\begin{split}
\Lambda^2_q V := \sum_{i\in\Z} \{ w \in V^{\otimes 2} | (R_V + q^i)^N w =0 \text{ for } N\gg0\}, \\\Gamma^2_q(V) := \sum_{i\in \Z}\{ w \in V^{\otimes 2} | (R_V-q^i)^N w=0  \text{ for } N\gg0\}.
\end{split}
\end{equation}

We can now define
\[ S^d_q(V) := \frac{V^{\otimes d} }  { \sum_{1 \leq i \leq d-1} V \otimes \cdots \otimes \Lambda^2(V)_{i,i+1} \otimes \cdots \otimes V}, \]
\[ \Lambda^d_q(V) := \cap_{1 \leq i \leq d-1} V \otimes \cdots \otimes \Lambda^2(V)_{i,i+1} \otimes \cdots \otimes V, \] 
 \[ \Gamma_q^d(V) := \cap_{1 \leq i \leq d-1} V \otimes \cdots \otimes \Gamma_q^2(V)_{i,i+1} \otimes \cdots \otimes V. \]

\begin{rem}
When $q$ is generic, our quantum divided power agrees with the definition of quantum symmetric power in \cite{BZ}. Note that semisimplicity makes it unnecessary for \cite{BZ} to distinguish the two. We prefer to define the quantum symmetric power as a quotient since it is more natural. %Note that for generic $q$ there is a (noncanonical) isomorphism between our definition of $S^d_q$ and of $\Gamma^d_q$, therefore our definition and the one in \cite{BZ} for quantum symmetric power correspond to the same object. 
\end{rem}

%Recall that $T_i \in \mathcal{B}_d$ acts on $V^{\otimes d}$ as $(R_V)_{i, i+1}$. 
%Define $M_i(V)$ to be the vector space spanned by the subspaces $M_i^r(V)= \operatorname{ker} ((R_V)_{i, i+1} - q^r)$ for all $r\in \mathbb{Z}$ and $N_i(V)$ the vector space spanned by the subspaces $N_i^r(V) = \operatorname{ker} ((R_V)_{i,i+1}+q^r)$ for all $r \in \mathbb{Z}$. Since $M_i^r(V)$ and $N_i^r(V)$ are defined as the kernel of an $A_q(n,n)$-comodule homomorphism, they are subcomodules of $V^{\otimes d}$. It follows that $M_i(V)$ and $N_i(V)$ are subcomodules of $V\td$. 

%Define 
%\[ S^d(V) = \cap_{1 \leq i \leq n-1} M_i(V), \]
%\[ \Lambda^d(V) = \cap_{1 \leq i \leq n-1} N_i(V). \] 
%Then $S^d$ and $\Lambda^d$ are the $d$-th symmetric and exterior powers, respectively. They are quantum polynomial functors that map $e$-Hecke pairs $V$ to $S^d(V)$ (resp, $\Lambda^d(V)$), a subcomodule of $V\td$. The action on morphisms is just the restriction of the map $f \in \Hom_{\mathcal{B}_d}(V\td, W\td)$ to the subspace $S^d(V)$ (resp., $\Lambda^d(V)$) whose image lives in $S^d(W)$ (resp., $\Lambda^d(W)$) due to the following lemma:

\begin{prop}\label{lem:genericqsymextpolfunct}
The quantum symmetric, exterior and divided powers are quantum polynomial functors. 
\end{prop}
\begin{proof}

Let $f\in \Hom_{\mathcal{B}_d} (V^{\otimes d}, W^{\otimes d})$. In order to show that $\Lambda^d_q$ and $\Gamma^d_q$ are quantum polynomial functors we need to show that the restriction of $f$ to $\Lambda^d_q (V)$ has image in $\Lambda^d_q (W)$ and the restriction of $f$ to $\Gamma^d_q (V)$ has image in $\Gamma^d_q (W)$. Since $f$ commutes with the action of $T_i \in \mathcal{B}_d$, we have $ f(V \otimes \cdots \otimes \Lambda_q^2(V)_{i,i+1} \otimes \cdots \otimes V)\subseteq W \otimes \cdots \otimes \Lambda_q^2(W)_{i,i+1} \otimes \cdots \otimes W$. It then easily follows from the definition that $f$ restricts to a map from $\Lambda^d_q (V)$ to $\Lambda^d_q (W)$. A similar proof works for $\Gamma^d_q$.

From the fact that $f$ maps  $ V \otimes \cdots \otimes \Lambda_q^2(V)_{i,i+1} \otimes \cdots \otimes V$ to $W \otimes \cdots \otimes \Lambda_q^2(W)_{i,i+1} \otimes \cdots \otimes W$ it follows that $f$ induces a map (which we denote by the same letter) $f:S^d_q (V) \to S_q^d(W)$ from which we deduce that $S^d_q$ is a quantum polynomial functor.  
\end{proof}

\begin{rem}
Symmetric/divided power and quantum symmetric/divided power are different quantum polynomial functors. % $\operatorname{char}(k)=0$. 
Consider for example the divided power functor $\Gamma^{2,{V_1^{\otimes e}}}_{q,e}$, where $V_1$ is the defining $A_q(1,1)$-comodule. For any $V\in \Gamma^2_{q,e}$, the image $\Gamma^{2,{V_1^{\otimes e}}}_{q,e}(V)$ is the eigenspace of $R_V:V\otimes V \to V \otimes V$ with eigenvalue $q^e$. The quantum divided power $\Gamma^2_q$ maps $V$ to the direct sum of eigenspaces of $R_V$ with eigenvalues $+q^r$ for all integers $r$. 
\end{rem}

\begin{rem}
Note that when $q=1$, the quantum polynomials $S^{d}_{V_1^{\otimes e}}$, $S^{d}_q$, $\Gamma^{d,V_1^{\otimes e}}_{q,e}$, and $\Gamma_q^d$ all return the classical symmetric power, since in that case $q^e=1=+q^r, \, \forall r$.
\end{rem}

With our definition, the following statement is clear.
\begin{prop}\label{prop:deg2ses}
Let $V$ be an $e$-Hecke pair. Let $N$ be the maximal rank of the Jordan blocks of $R_V$.  
The sequence
 \[0 \to \Gamma^2_q V  \xrightarrow{p_1} V^{\otimes 2} \xrightarrow{p_2} \Lambda^2 V \to 0\]
is exact, where $p_1$ is the inclusion $\Gamma^2_q V \xhookrightarrow{} V^{\otimes 2}$ (recall that we defined $\Gamma^2_q V$ as a subspace of $V^{\otimes 2}$) and $p_2$ is defined by
\[ p_2: w \mapsto \prod_{-e\leq i\leq e} (R_V-q^i)^Nw .\] 
\end{prop}
%\begin{proof}
%Exactness the first and the second place follows easily by the definitions of the objects in the sequence. The only thing that is not clear is that the image of $p_2$ is inside $\Lambda^2 V$. But it follows from the fact that $V^{\otimes 2}$ splits as the direct sum of $\Lambda^2 V$ and $\Gamma^2_q V$. \end{proof}

Here we record two short exact sequences of degree $2$ quantum polynomial functors
\begin{equation}\label{deg2ses}
0\to\Lambda^2_q\to\otimes^2\to S^2_q\to 0,
\end{equation}
which is by definition, and
\begin{equation}\label{deg2ses2}
0\to\Gamma^2_q\to\otimes^2\to \Lambda^2_q\to 0,
\end{equation}
which follows from Proposition \ref{prop:deg2ses}.
Also note that the first map in \eqref{deg2ses} and the last map in \eqref{deg2ses2} make an exterior power a direct summand of a tensor power.

Furthermore, we have the following.

\begin{prop}
There are isomorphisms of quantum polynomial functors
\[ (\otimes^n)^\#\cong\otimes^n,\]
\[ (S_q^n )^\# \cong \Gamma^n_q ,\]
\[ (\Lambda^n_q)^\# \cong \Lambda^n_q .\]

\end{prop}
\begin{proof}
Let $V$ be an $e$-Hecke pair, or equivalently, a polynomial representation of $\uqgln$ of degree $e$.
There is a canonical isomorphism 
\[\phi_n: (^\tau V )^{\otimes n}\to ^\tau\!\! (V^{\otimes n})\] of $\uqgln$-modules, since $\tau_1$ and the comultiplication commute. (See \cite[9.20]{quantumJ}. We note that what Jantzen denotes by the twisted comultiplication $\Delta'$ is the usual comultiplication structure that makes $A_q(n,n)$ live in the dual of $U_q(\mathfrak{gl}_n)$.) 
This is then also an isomorphism of $A_q(n,n)$-comodules. 
By \eqref{doubledual}, the dual of $\phi_n$ induces the desired isomorphism $\otimes^n\cong(\otimes^n)^\#$ of polynomial functors.

To prove the other statements, first consider the $n=2$ case.
The $R$-matrix of $^\tau V$ satisfies 
\[R_{^\tau V} = R_V.\] 
This follows from the fact that the $R$-matrix of $V^*$ is the transpose of $R_V$ (follows easily from Proposition 4.2.7 in \cite{CP}) and the matrix of $V$ twisted by $\tau_1$ is also the transpose of $R_V$ (follows from the formulas in \cite[9.20]{quantumJ}). The isomorphism $\phi_2$ preserves eigenspaces of $R_{^\tau V} = R_V$, and hence restricts to an isomorphism from $(\Lambda^2_q (^\tau V))$, which is the direct sum of eigenspaces of $R_{^\tau V}$ corresponding to eigenvalues $-q^i$, to $^\tau\!(\Lambda^2_q V)$, which is the direct sum of eigenspaces of $R_{V}$ corresponding to eigenvalues $-q^i$. 
Thus, $\phi_2$ induces an isomorphism \[\Lambda^2_q (^\tau V) \cong ^\tau\!\!\!(\Lambda^2_q V).\] 
The duality $^\tau$ satisfies $^\tau (V\otimes W) \cong ^\tau \! V \otimes ^\tau W$ because both $\tau_1$ and the antipode are antiautomorphisms.
So we have
\[^\tau \!(V \otimes \cdots \otimes \Lambda_q^2(V)_{i,i+1} \otimes \cdots \otimes V) \cong ^\tau \!\!V \otimes \cdots \otimes \Lambda_q^2(^\tau V)_{i,i+1} \otimes \cdots \otimes ^\tau \!\! V.\] 
By intersecting we obtain $^\tau \! (\Lambda^d_q(V)) \cong \Lambda_q^d(^\tau V)$, from which $\Lambda^n_q \cong (\Lambda^n_q)^\# $ follows (again, we use \eqref{doubledual}).

The second statement is then obtained by comparing the short exact sequence \eqref{deg2ses} and the dual of the short exact sequence \eqref{deg2ses2} for $d=2$. 
Then it is clear from the definition of the quantum symmetric and divided powers that the same is true for all $d$.
\end{proof}

%\subsubsection{Some remarks on symmetric and exterior power at roots of unity} 

%The construction of the symmetric and exterior power for generic $q$ uses a crucial fact, namely that the category of $A_q(n,n)$-comodules is semisimple. This does not hold when $q$ is a root of unity even when the field $k$ has characteristic 0. In this subsection we fix $q$ to be a root of unity in $\mathbb C$ and note some remarks about how one should think about the symmetric and exterior power. 

\begin{rem}
We now explain why we assume $p$ to be an odd root of unity. Assume $q=i=\sqrt{-1}$ and $V_2$ is the defining $A_q(2,2)$-comodule. We want to decompose $V_2^{\otimes 2}$ into two parts $\Gamma_q^2(V_2)$ and $\Lambda_q^2(V_2)$. 
The problem with the construction above is that the $R$-matrix $R_2$ only has one eigenvalue $q = -q^{-1} = i$. It is not diagonalizable; it has a $2 \times 2$ Jordan block and two $1 \times 1$ blocks. 
Thus, we cannot separate the eigenvalues of $R_{V_2}$ into ``positive'' and ``negative'' eigenvalues and therefore definition \eqref{eq:extanddiv} doesn't make sense.  
\end{rem} 

%\begin{rem}
%We remark that the definition of symmetric and exterior power for $1$-Hecke pairs in \cite[Section 3.3]{HongYacobi} applied to $q=i$ (which they exclude) says that both $S^2(V_2)$ and $\Lambda^2(V_2)$ are equal and that $v_1 \otimes v_1 \notin S^2(V_2)$. For generic $q$, their definition agrees with \cite{BZ}. 
%\end{rem}

%\begin{rem}
%There is no definition that generalizes the one in \cite{BZ} that does not depend on choosing an eigenbasis of $V\otimes V$, as far as the authors know. The problem is that the decomposition of $R_V$ into Jordan blocks in the root of unity case is not well understood. 
%The categorical symmetric powers are defined in Remark \ref{rem:categoricalsympower} naturally in all cases, but none of them agree with the symmetric power in \cite{BZ} (for generic $q$). Moreover, it is unclear how one should define the categorical exterior power. 
%\end{rem}

\section{Braiding on $\mathcal{P}_{q,e}$}\label{section:braiding}
We can use the braiding structure on the category of $A_q(n,n)$-comodules to endow the category $(\mathcal{P}_{q, e},\otimes)$ with the structure of a braided monoidal category.

Let $F \in \mathcal{P}_{q,e}^{d}$ and $G \in \mathcal{P}_{q,e}^{d'}$. The tensor products $F \otimes G$ and $G \otimes F$ both live in $\mathcal{P}_{q,e}^{d + d'}$.
For an $e$-Hecke pair $V$, recall that $F(V)$ and $G(V)$ are both $A_q(n,n)$-comodules for some $n$ according to Proposition \ref{prop:composition}. 
Let $R_{F(V), G(V)} : F(V) \otimes G(V) \to G(V) \otimes F(V)$ be the braiding isomorphism defined in equation \eqref{eq:Rmatrixcomodule}. We use it to define the natural transformation $R_{F,G} : F \otimes G \to G \otimes F$ by \[R_{F,G} (V) := R_{F(V), G(V)}.\] 
This map turns $\mathcal{P}_{q,e}$ into a braided monoidal category. But before we can prove that we need two results which are interesting in their own right.

Let $V$ be an $e$-Hecke Pair and let $F \in \pde$ and $G \in \mathcal{P}^{d'}_{q,e}$. Then $F(V)$ and $G(V)$ are naturally $A_q(V,V)$-comodules. They also have the structure of $A_q(n,n)$-comodules by Proposition \ref{prop:composition}.

The coalgebra $A_q(n,n)$ is coquasitriangular, therefore the comodule structure supplies us with the braiding $R_{F(V), G(V)}:F(V) \otimes G(V) \to G(V) \otimes F(V)$ that we mentioned above. However the coalgebra $A_q(V,V)$ is also coquasitriangular. Therefore there is a universal $R$-matrix $\mathcal{R}^V \in A_q(V \otimes V) \otimes A_q(V\otimes V)$ that satisfies properties \eqref{equation:cqt} and is defined on the generators of $A_q(V,V)$ by
\[ \mathcal{R}^V (x^V_{ij} \otimes x^V_{kl}) = (R_V)^{ik}_{lj}.\] 
Since $F(V), G(V)$ are $A_q(V,V)$-comodules, there is an $R$-matrix $R_{F(V), G(V)}': F(V) \otimes G(V) \to G(V) \otimes F(V)$.

\begin{prop}\label{lemma:braiding1}
The maps $R_{F(V), G(V)}$ and $R'_{F(V), G(V)}$ are equal. 
\end{prop}
\begin{proof}
The map $\psi_V : A_q(V,V) \to A_q(n,n)$  defined in equation \eqref{eq:psiV} has the property that 
\[\mathcal{R}^V (x, y) = \mathcal{R}(\psi_V(x), \psi_V(y))\]
for all $x,y \in A_q(V,V)$. The equation holds when $x,y$ are generators of $A_q(V,V)$ because we defined $\mathcal{R}' (x^V_{ij} \otimes x^V_{kl}) = (R_V)^{ik}_{lj}$.  
The $R$-matrix of the $A_q(n,n)$-comodule $V$ is $R_V$, therefore the following equation holds by definition $\mathcal{R}(t^V_{ij} \otimes t^V_{kl}) = (R_V)^{ik}_{lj}$. 
We have that $\psi_V (x^V_{ij}) = t_{ij}^V$ and therefore the equation above holds when $x,y$ are generators. 

Since the equation holds on generators, it holds for every $x,y \in A_q(V,V)$. It implies that the $R$-matrices $R_{F(V), G(V)}$ and $R'_{F(V), G(V)}$ are equal. 
\end{proof}

\begin{lem}\label{lemma:braiding2}
The following diagram is commutative:
\[ \begin{tikzpicture} 
  \matrix (m) [matrix of math nodes,row sep=3em,column sep=13em,minimum width=2em]
  {
     F(V) \otimes G(V) & F(W) \otimes G(W) \otimes A_q(W,V) \otimes A_q(W,V) \\
     G(V) \otimes F(V) & G(W) \otimes F(W) \otimes A_q(W,V) \otimes A_q(W,V) \\};
  \path[-stealth]
    (m-1-1) edge node [left] {$R_{F(V), G(V)}$} (m-2-1)
            edge node [above] {$(1 \otimes \tau \otimes 1) (F_{V,W}'' \otimes G_{V,W}'')$} (m-1-2)
    (m-2-1.east|-m-2-2) edge node [below] {$(1 \otimes \tau \otimes 1)( G_{V,W}'' \otimes F_{V,W}'')$} (m-2-2)
    (m-1-2) edge node [right] {$R_{F(W), G(W)} \otimes 1$} (m-2-2) ;
\end{tikzpicture}\]
\end{lem}
\begin{proof}
Note that by Proposition \ref{lemma:braiding1}, we can replace in the equation above $R_{F(V),G(V)}$ and $R_{F(W),G(W)}$ by $R'_{F(V),G(V)}$ and $R'_{F(W),G(W)}$, respectively. 

By considering the commutative diagram in equation \eqref{eq:diagramGWV} in the proof of Lemma \ref{psigeneral} for the functor $F \oplus G$ we obtain that the diagram 
\[\begin{CD}
(F \oplus G) (V) \otimes (F \oplus G) (V)@>{(1\otimes \tau \otimes 1)((F \oplus G)_{V,W}'' \otimes (F \oplus G)_{V,W}'')}>>(F \oplus G)(W) \otimes (F \oplus G) (W) \otimes  A_q(W,V)\\
@V{R_{F \oplus G(V)}}VV @V{R_{F \oplus G(W)}\otimes 1}VV\\
(F \oplus G) (V) \otimes (F \oplus G)(V)@>{(1\otimes \tau \otimes 1)((F \oplus G)_{V,W}'' \otimes (F \oplus G)_{V,W}'')}>> (F \oplus G)(W) \otimes (F \oplus G) (W) \otimes  A_q(W,V)
\end{CD}\]
is commutative. The result follows from the fact that the tensor product $(F \oplus G) (V) \otimes (F \oplus G)(V)$ can be written as 
\[F(V) \otimes F(V) \oplus F(V) \otimes G(V) \oplus G(V) \otimes F(V) \oplus G(V) \otimes G(V)\]
and that the restriction of $R_{F \oplus G(V)}$ to $F(V) \otimes G(V)$ is just $R_{F (V), G(V)}$, while the restriction of $(1\otimes \tau \otimes 1)((F \oplus G)_{V,W}'' \otimes (F \oplus G)_{V,W}'')$ to $F(V) \otimes G(V)$ and $G(V) \otimes F(V)$ is just $(1 \otimes \tau \otimes 1) F_{V,U}'' \otimes G_{V,U}''$ and $(1 \otimes \tau \otimes 1) (G_{V,U}'' \otimes F_{V,U}'')$, respectively.

\end{proof}

\begin{prop}\label{lem:naturalitytocomodule}
Let $F, F' \in \pde$ be quantum polynomial functors and let $\alpha: F \to F'$ be a natural transformation. Then the maps $\alpha_V : F(V) \to F'(V)$ are $A_q(V,V)$-comodule homomorphisms. 
\end{prop}

\begin{proof}
The following diagram commutes for any $f\in \Hom(V,V)\cong S_q(V,V;d)$ because $\alpha$ is a natural transformation:
\[ \begin{tikzpicture} 
  \matrix (m) [matrix of math nodes,row sep=3em,column sep=5em,minimum width=2em]
  {
     F(V)  & F'(V) \\
     F(V) & F'(V) \\};
  \path[-stealth]
    (m-1-1) edge node [left] {$F(f)$} (m-2-1) 
            edge node [above] {$\alpha_V$} (m-1-2)
    (m-2-1.east|-m-2-2) edge node [below] {$\alpha_V$} (m-2-2)
    (m-1-2) edge node [right] {$F'(f)$} (m-2-2) ;
\end{tikzpicture}\]
This is the same as saying that the map $\alpha_V$ is an $S_q(V,V;d)$-module homomorphism from which the conclusion follows. 
\end{proof}

Now we can prove the main result of this subsection.

\begin{thm}\label{thm:tensorproductisbraided}
The category $\mathcal{P}_{q,e}$ is a braided monoidal category with braiding isomorphism $R_{F,G}: F \otimes G \to G \otimes F$. 
\end{thm}
\begin{proof}
%It is a natural transformation because $R_{F,G} (V) = R_{F(V), G(V)}$ is an $A_q(n,n)$-comodule homomorphism for every $V$, and 
We first show that the braiding is a well defined morphism in $\Hom (F \otimes G, G \otimes F)$. This is equivalent to showing the commutativity of the following diagram 
\[ \begin{tikzpicture} 
  \matrix (m) [matrix of math nodes,row sep=3em,column sep=5em,minimum width=2em]
  {
     F(V) \otimes G(V) & F(W) \otimes G(W) \\
     G(V) \otimes F(V) & G(W) \otimes F(W) \\};
  \path[-stealth]
    (m-1-1) edge node [left] {$R_{F(V), G(V)}$} (m-2-1) 
            edge node [above] {$(F\otimes G) (f)$} (m-1-2)
    (m-2-1.east|-m-2-2) edge node [below] {$(G \otimes F) (f)$} (m-2-2)
    (m-1-2) edge node [right] {$R_{F(W), G(W)}$} (m-2-2) ;
\end{tikzpicture}\]
for all $e$-Hecke pairs $V, W$ and any $f \in \Hom (V, W)$. The commutativity of the diagram above for all $f$ is equivalent to the commutativity of the following diagram 
\[ \begin{tikzpicture} 
  \matrix (m) [matrix of math nodes,row sep=3em,column sep=5em,minimum width=2em]
  {
     S_q(V,W;d+d') \otimes F(V) \otimes G(V) & F(W) \otimes G(W) \\
     S_q(V,W;d+d') \otimes G(V) \otimes F(V) & G(W) \otimes F(W) \\};
  \path[-stealth]
    (m-1-1) edge node [left] {$1 \otimes R_{F(V), G(V)}$} (m-2-1)
            edge node [above] {} (m-1-2)
    (m-2-1.east|-m-2-2) edge node [below] {} (m-2-2)
    (m-1-2) edge node [right] {$R_{F(W), G(W)}$} (m-2-2) ;
\end{tikzpicture}\]
where the horizontal maps are given by $f \otimes v^F \otimes v^G \mapsto (F \otimes G)(f) (v^F \otimes v^G)$. %Note that we are abusing notation and we think of $f \in S_q(V,W;d+d')$ as its image in $\Hom (V, W)$. 

The commutativity of the second diagram is now equivalent to the commutativity of the following diagram by the definition of $S_q(V,W;d+d')$ as the dual of $A_q(W,V)_{d+d'}$. 
\[ \begin{tikzpicture} 
  \matrix (m) [matrix of math nodes,row sep=3em,column sep=11em,minimum width=2em]
  {
     F(V) \otimes G(V) & F(W) \otimes G(W) \otimes A_q(W,V) \otimes A_q(W,V) \\
     G(V) \otimes F(V) & G(W) \otimes F(W) \otimes A_q(W,V) \otimes A_q(W,V) \\};
  \path[-stealth]
    (m-1-1) edge node [left] {$R_{F(V), G(V)}$} (m-2-1)
            edge node [above] {$(1\otimes \tau \otimes 1) F_{V,U}'' \otimes G_{V,U}''$} (m-1-2)
    (m-2-1.east|-m-2-2) edge node [below] {$(1\otimes \tau \otimes 1)G_{V,U}'' \otimes F_{V,U}''$} (m-2-2)
    (m-1-2) edge node [right] {$R_{F(W), G(W)} \otimes 1$} (m-2-2) ;
\end{tikzpicture}\]
which is commutative by Lemma \ref{lemma:braiding2}. This completes the proof of naturality of the braiding. 

Now we show that $R_{F, G}$ is a natural transformation, namely we need to show that for $f: F \to F'$ and $g: G \to G'$ we have \[ R_{F',G'} (f\otimes g) = (g\otimes f) R_{F,G} \]
The relation above holds when applied to any $V$ because we can write $R_{F,G}(V) = R_{F(V), G(V)}$ as the composition $(1 \otimes 1 \otimes \mathcal{R}^V)(1 \otimes \tau \otimes 1)(\Delta_{G(V)} \otimes \Delta_{F(V)})\tau$ and $f_V \otimes g_V$ commutes with each factor of that composition except for $\tau$ which switches $f_V$ and $g_V$ by Proposition \ref{lem:naturalitytocomodule}. Notice that in the equation above we use $\mathcal{R}^V$, the universal $R$-matrix of the coalgebra $A_q(V,V)$, instead of $\mathcal{R}$. We can do this because of Proposition \ref{lemma:braiding1}. 

The natural transformation $R_{F,G}$ is an isomorphism because $R_{F(V), G(V)}$ is an isomorphism for every $V$.
We show that it satisfies equations \eqref{eq:braidedmonoidalcategory}. To prove the first property $\gamma_{V \otimes  W, U} =  (\gamma_{V,U}\otimes 1) (1 \otimes \gamma_{W, U})$ we note that it is equivalent to  
\[ R_{F \otimes  G, H}(V) =  (R_{F,H}(V)\otimes 1_G) (1_F \otimes R_{G, H}(V))\]
for every $e$-Hecke pair $V$. This can be rewritten as 
\[ R_{F(V) \otimes  G(V), H(V)} =  (R_{F(V),H(V)}\otimes 1_{G(V)}) (1_{F(V)} \otimes R_{G(V), H(V)})\]
and follows immediately from Proposition \ref{prop:Aqnnbraidedmonoidalcategory}.

The third property $\tilde{r}_V \gamma_{I, V} = \tilde{l}_V$ can be rewritten as 
\[ r_V R_{k, F(V)} = l_{F(V)} \]
which again follows immediately from Proposition \ref{prop:Aqnnbraidedmonoidalcategory}.

The rest of the properties properties follow by the same argument as above.   
\end{proof}

\begin{rem}
One can similarly show the existence of a braiding for the monoidal category $\mathcal{P}^{\circ, d}_{q,e}$.
\end{rem}

%\begin{rem}
%Note that Hong and Yacobi \cite[Section 5]{HongYacobi} showed that the category $\mathcal{P}_{q,1}$ (see Remark \ref{eqtoHY}) is braided using the equivalence between $\mathcal{P}_{q,1}$ and $\operatorname{mod} (S_q(n,n;d))$. Our result does not use the braided structure on $\operatorname{mod} (S_q(n,n;d))$, but rather the (more basic, but equivalent) braided structure on $\operatorname{comod} (A_q(n,n))$. This allows us to endow a braiding structure on the category of modules over the $(q,e)$-Schur algebras $S_q(V_n^{\otimes e},V_n^{\otimes e},d)$, using the equivalence of categories $\mathcal{P}^d_{q,e} \cong \operatorname{mod}(S_q(V_n^{\otimes e},V_n^{\otimes e},d))$
%which we prove in Corollary \ref{cor:equivalenceofcategories}.
%@@@question about teh last sentence: aren't the q-e-schur algebras already quasitriangular? the map R can be defined in the same way as in (8). yes? @@@
%\end{rem}

\section{Composition of quantum polynomial functors}\label{section:composition}

Given two linear functors $F,G$ between arbitrary $k$-linear categories, one can define the composition $F\circ G$ if the domain of $F$ agrees with the codomain of $G$, and then $F\circ G$ is a linear functor. 
The quantum polynomial functors, as in the definition presented, have (up to equivalence) $\mathcal V$ as their codomain, which does not match the domains of quantum polynomial functors. We have to endow the image of $F$ with some additional structure.

%we lift a linear functor
%\[ F : \Gamma^d_{q,e} \mathcal{V} \to \Gamma^1_{q,ed}\mathcal{V} \]
%to a class of linear functors
%\[ F : \Gamma^{dd'}_{q,e} \mathcal{V} \to \Gamma^{d'}_{q,ed}\mathcal{V} \]
%where $d'$ runs through the positive integers. 
%This section works this out (in a more elementary language).
%@@@the intro and the content doesn't match too well but I think the idea mentioned in the intro is a good formulation of the main problem in doing the composition@@@

\begin{prop} \label{prop:composition}
Given $V$ an $e$-Hecke pair and $F$ a quantum polynomial functor
of degree $d$, then $F(V)$ is also an $A_q(n,n)$-comodule. It is a $de$-Hecke pair. 
\end{prop}
\begin{proof}

By Lemma $\ref{comodule2}$, $F(V)$ is an $A_q(V,V)$-comodule with coaction 
$\Delta^V_{F(V)}$. Since $\psi_V$ is a coalgebra homomorphism by Lemma \ref{psiV}, the composition 
$(1 \otimes  \psi_V) \Delta^V_{F(V)}$ makes $F(V)$ into an $A_q(n,n)$-comodule. 
It is easy to see from the way $(1 \otimes  \psi_V) \Delta^V_{F(V)}$ is defined that $F(V)$ is a $de$-Hecke pair. 
This completes the proof of the statement. 
\end{proof}

We now explain the composition of quantum polynomials functors.
Let $G \in \mathcal{P}^{d_2}_{q, e}$ and $F \in \mathcal{P}^{d_1}_{q,d_2 e}$. We define $F \circ G \in \mathcal{P}_{q,e}^{d_1 d_2}$
as follows: on objects $V \in \Gamma^{d_1 d_2}_{q,e} \mathcal{V}$ we let 
\[(F \circ G) (V) = F(G(V)).\] 
This composition makes sense because $G(V)$ is an $A_q(n,n)_{d_2 e}$-comodule by Proposition \ref{prop:composition}. Since $F$  is a quantum polynomial functor of degree $d_1$, we have maps 
\[ F_{G(V),G(W)}'' : F(G(V)) \to F(G(W)) \otimes A_q(G(W),G(V))_{d_1} \]
that satisfy the commutation relations in Proposition \ref{equivalentcharacterization}. $G$ is also a quantum polynomial functor so we have maps:
\[ G_{V,W}'' : G(V) \to G(W) \otimes A_q(W,V)_{d_2}. \]

Define $(F \circ G)_{V,W}'' : F(G(V))\to F(G(W))\otimes A_q(W,V)$ as 
\[(F \circ G)_{V,W}'' := (1 \otimes \psi^G_{V,W}) \circ F_{G(V),G(W)}''\]
where $\psi_{V,W}^G$ is defined in Lemma \ref{psigeneral}.

\begin{thm}\label{thm:composition}
$(F \circ G)''$ satisfies properties \eqref{polyfunctor1}, \eqref{polyfunctor2} in Proposition \ref{equivalentcharacterization}. Therefore 
$F \circ G$ is a well-defined quantum polynomial functor in $\mathcal{P}_{q,e}^{d_1 d_2}$.
\end{thm}
\begin{proof}

Diagram \eqref{polyfunctor1} for $(F \circ G)$ is equivalent to the exterior square of the following diagram:

\hspace{-1cm}
\footnotesize
\begin{tikzpicture}
  \matrix (m) [matrix of math nodes,row sep=3em,column sep=3em,minimum width=3em]
  {
     (F\circ G)(V) & (F \circ G)(U) \otimes A_q(G(U),G(V)) & (F \circ G)(U) \otimes A_q(U, V) \\
     (F\circ G)(W) \otimes A_q(G(W),G(V)) & (F\circ G)(U) \otimes A_q(G(U),G(W)) \otimes A_q(G(W),G(V)) &  \\
     (F\circ G)(W) \otimes A_q(W,V) &  & (F\circ G)(U) \otimes A_q(U,W) \otimes A_q(W,V) \\};
  \path[-stealth]
    (m-1-1) edge node [left] { $F_{G(V),G(W)}''$} (m-2-1)
            edge node [above] {$F_{G(V),G(U)}''$} (m-1-2)
    (m-2-1.east|-m-2-2) edge node [below] {$F_{G(W),G(U)}''\otimes 1$} (m-2-2)
    (m-1-2) edge node [right] {$1 \otimes \Delta_{G(U),G(W),G(V)}$} (m-2-2)
    (m-2-2) edge node [midway,above] {\,\,\,\,\,\,\,\,\,\,\,\,\,$1 \otimes \psi_{U,W}^G \otimes \psi_{W,V}^G$} (m-3-3)
    (m-1-2) edge node [above] {$1 \otimes \psi_{U,V}^G$} (m-1-3)
    (m-2-1) edge node [left] {$1 \otimes \psi_{W,V}^G$} (m-3-1)
    (m-1-3) edge node [right] {$1 \otimes \Delta_{U,W,V}$} (m-3-3)
    (m-3-1) edge node [below] {$(F\circ G)_{W,U}'' \otimes 1$} (m-3-3);
\end{tikzpicture}
\normalsize

We show the commutativity of the exterior square by showing the commutativity of all three interior quadrilaterals. The commutativity of the top left diagram follows from the fact that $F$ is a quantum polynomial functor. The bottom left diagram follows from the fact that the horizontal arrows modify only the left component of the tensor product, while the vertical arrows modify the right component only. The commutativity of the top right diagram follows from the fact that $\psi^G_{U,V}$ satisfies equation \eqref{eq:psiG1}. 

Diagram \eqref{polyfunctor2} for $F\circ G$ is the exterior triangle of the following diagram:

\small
\begin{center}
\begin{tikzpicture}
  \matrix (m) [matrix of math nodes,row sep=5em,column sep=8em,minimum width=2em]
  {
     (F\circ G)(V) & (F\circ G)(V) \otimes A_q(G(V),G(V)) & (F\circ G)(V)\otimes A_q(V,V)\\
     (F\circ G)(V)\otimes k  & &  \\};
  \path[-stealth]
    (m-1-1) edge node [above] {$F_{G(V),G(V)}''$} (m-1-2)
    (m-1-1) edge node [left] {$1$} (m-2-1)
    (m-1-2) edge node [above]{\hspace{-2cm}$1 \otimes e_{A_q(G(V),G(V))}$} (m-2-1)
    (m-1-2) edge node [above] {$1\otimes \psi^G_{V,V}$} (m-1-3)
    (m-1-3) edge node [below] {\hspace{3cm}$1\otimes e_{A_q(V,V)}$} (m-2-1);
\end{tikzpicture}
\end{center}
\normalsize

The left triangle commutes by \eqref{polyfunctor2} and the right triangle by \eqref{eq:psiG2}.

\end{proof}

\section{Representability}\label{sec:finitegeneration}
Recall the divided power functor $\Gamma_{q,e}^{d,V}\in \pde$:
\[\Gamma_{q,e}^{d,V}(W)=\Hom_{\mathcal{B}_d}(V^{\otimes d},W^{\otimes d})\] defined for each $e$-Hecke pair $V\in \Gamma^d_{q,e}\mathcal V$. 
%on objects and as the obvious composition on morphisms. 

\begin{prop}\label{prop:projective}(Yoneda's Lemma)
For any $e$-Hecke pair $W$, the functor $\Gamma^{d,W}_{q,e} \in \mathcal{P}^d_{q,e}$ represents the evaluation functor $\mathcal{P}^d_{q,e} \to \mathcal{V}$ given by $F \mapsto F(W)$; therefore $\Gamma^{d,W}_{q,e}$ is a projective object in $\mathcal{P}^d_{q,e}$. 
\end{prop}
\begin{proof}
We need to show the existence of an isomorphism 
\[ \textnormal{Hom}_{\mathcal{P}^d_{q,e}}(\Gamma^{d,W}_{q,e}, F ) \to F(W)  \]
for any $F \in \mathcal{P}^d_{q,e}$. 
Define the map $\rho :  \textnormal{Hom}_{\mathcal{P}^d_{q,e}}(\Gamma^{d,W}_{q,e}, F ) \to F(W)$ by 
\[ \rho (f) = f_{W} (\textnormal{id}) \in F(W). \]
Let $\phi: F(W) \to  \textnormal{Hom}_{\mathcal{P}^d_{q,e}}(\Gamma^{d,W}_{q,e}, F )$ be defined as follows: for any $v \in F(W)$, $\phi(v)$ is the natural transformation such that $\phi(v) (U) :\Gamma^{d,W}_{q,e} (U) \to F(U)$ takes 
\[ g \in \textnormal{Hom}_{\mathcal{B}_d} ( (W)^{\otimes d}, U^{\otimes d}) \mapsto F_{W_{n}, U} (g) (v) \in F(U). \]

We now show that $\rho$ and $\phi$ are inverses to each other. Start with $v \in F(W)$. We have 
\[\rho(\phi(v)) = \phi(v) (W)(\textnormal{id}) = F_{W,W} (\textnormal{id}) (v) = v \]

Now let $f \in \textnormal{Hom}_{\mathcal{P}^d_{q,e}}(\Gamma^{d,W}_{q,e}, F )$. We want to show that \[\phi(\rho(f))_U(g)=f_U(g)\] for any $U\in \Gamma^d_{q,e}$ and $g\in \textnormal{Hom}_{\mathcal{B}_d} ( (W)^{\otimes d}, U^{\otimes d})$.
\begin{equation*}
\begin{split}
\phi(\rho(f))(U)(g)&=F_{W,U}(g)(f_{W}(\textnormal{id}))\\
&=f_U(\Gamma^{d,W}_{q,e}(g)(\textnormal{id}))\\
&=f_U(g).
\end{split}
\end{equation*}
The second equality follows from the commutativity of the following diagram 

\begin{equation}\label{projectiveobject}
\begin{tikzpicture} 
  \matrix (m) [matrix of math nodes,row sep=3em,column sep=11em,minimum width=2em]
  {
     \Gamma^{d,W}_{q,e}(W)& \Gamma^{d,W}_{q,e}(U) \\
     F(W) & F(U) \\};
  \path[-stealth]
    (m-1-1) edge node [left] {$f_{W}$} (m-2-1)
            edge node [above] {$(\Gamma^{d,W}_{q,e})_{W,U} (g)$} (m-1-2)
    (m-2-1.east|-m-2-2) edge node [below] {$F_{W,U}(g)$} (m-2-2)
    (m-1-2) edge node [right] {$f_U$} (m-2-2) ;
\end{tikzpicture}
\end{equation}
which holds because $f$ is a natural transformation. 

It follows that $\Gamma^{d,W}_{q,e}$ is a projective object in $\mathcal{P}^d_{q,e}$. 
\end{proof}

\begin{defn}\label{fgdef}
Let $W$ be an $e$-Hecke pair. The quantum polynomial functor $F \in \mathcal{P}^d_{q,e}$ is $W$-$\it{generated}$ if
for every $e$-Hecke pair $U$ the map
\[ F_{W,U}' : S_q(W,U;d) \otimes F(W) \to F(U) \]
is surjective. We say that the category $\pde$ is $\it{finitely}$ $\it{generated}$ if there is an $e$-Hecke pair $W$ such that every $F\in\pde$ is $W$-generated. 
\end{defn}

\begin{rem}
Note that $S_q(W,U;d)=\Gamma^{d,W}_{q,e}(U)$. So the map $F_{W,U}'$ above gives a surjection 
\[\Gamma_{q,e}^{d,W}\otimes F(W)\to F\]
in $\pde$, where $\Gamma_{q,e}^{d,W}\otimes F(W)$ is interpreted as the direct sum of $\dim F(W)$ copies of $\Gamma_{q,e}^{d,W}$ in $\pde$. 
Since $\Gamma_{q,e}^{d,W}$ is projective by Proposition \ref{prop:projective}, Definition \ref{fgdef} says that $\Gamma_{q,e}^{d,W}$ is a projective generator of the category $\pde$. 
\end{rem}

\begin{defn}
Given two objects $V,W\in \Gamma^d_{q,e}\mathcal V $, we say $V$ generates $W$ if the identity $\Id _{W^{\otimes d}}$ (which should be written $\Id_W$ if we view it as an identity in the category $\Gamma_{q,e}^d\mathcal V$) can be written as a linear combination of $\mathcal{B}_d$-homomorphisms which factor through $V^{\otimes d}$ (which should be stated as `in $\Gamma_{q,e}^d\mathcal V$, the identity on $W$ can be written as a linear combination of endomorphisms on $W$ that factor through $V$').
\end{defn}

\begin{prop}\label{equicon}
Let $V, W\in \Gamma^d_{q,e}\mathcal V $. Assume that any indecomposable summand of $W^{\otimes d}$ (as a $\mathcal{B}_d$-module) is isomorphic to a direct summand of $V^{\otimes d}$. Then $V$ generates $W$.
\end{prop}
\begin{proof}
Denote by $M_i$ all the indecomposable summands appearing in a fixed indecomposable summand decomposition of $W^{\otimes d}$. For each $M_i$, there is a summand $N_i$ in $V^{ \otimes d}$ that is isomorphic to $M_i$ as $\mathcal{B}_d$-modules. Let $f_i$ be a $\mathcal{B}_d$-map that maps the summand $M_i \subset W^{\otimes d} \to N_i \subset V^{ \otimes d}$ and $g_i$ a $\mathcal{B}_d$-map that maps $N_i \subset V^{\otimes d} \to M_i \subset W^{\otimes d}$. Then the sum of $g_i \circ f_i$ is the identity on $W^{\otimes d}$. 
\end{proof}
\begin{lem}\label{generatorgivessurjective}
If $V$ generates $W$, then the natural map \[F(V)\otimes \Gamma_q^{d,V}(W)\to F(W)\] is surjective for any $F\in\pde$.
\end{lem}
\begin{proof}
Since there are maps $f_i, g_i$ such that $\sum_i g_i \circ f_i = \text{Id}_{W^{\otimes d}}$, by applying $F$ we obtain $\sum_i F(g_i) \circ F(f_i) = \text{Id}_{F(W)}$. This implies that the natural map $F(V) \otimes \text{Hom}_{\pde} (V,W) \to F(W)$ is surjective. 
\end{proof}

We obtain an important corollary.
\begin{cor}\label{gammapg}
The functor $\Gamma_q^{d,V}$ is a projective generator in $\pde$ %, that is, for each $F\in\pde$  
if $V$ generates $W$ for all $W\in \Gamma_{q,e}^d\mathcal V$.
\end{cor}
\begin{proof}
This follows immediately from the proof of Lemma \ref{generatorgivessurjective}.
\end{proof}

\begin{lem}\label{karoubi}
If $V$ generates $W$ then $V$ generates any direct summand of $W$ as an $A_q(n,n)$-comodule.
\end{lem}
\begin{proof}

If $V$ generates $W$ and $W= W_1 \oplus W_2$ (as $A_q(n,n)$-comodules), then $V$ generates $W_1$. To see this, assume the existence of maps $f_i, g_i$ such that $\sum_i g_i \circ f_i = \text{Id}_{W^{\otimes d}}$. Since $W_1$ is a direct summand, there are inclusion and projection maps $i: W_1 \to W_1 \oplus W_2 = W$ and $p:W = W_1 \oplus W_2 \to W_1$. Then $i^{\otimes d}:W_1^{\otimes d} \to W^{\otimes d}$ and $p^{\otimes d}:W^{\otimes d} \to W_1^{\otimes d}$ are $\mathcal{B}_d$-maps because $\mathcal{B}_d$ acts via $R$-matrices which are $A_q(n,n)$-maps. Let $\bar f_i: W_1^{\td} \to V^{\td}, \bar g_i : V^{\td} \to W_1^{\td}$ be defined as $\bar f_i = f_i \circ i^{\td}$ and $\bar g_i = p^{\td} \circ g_i$, respectively. Then $\bar f_i, \bar g_i$ are $\mathcal{B}_d$-maps and $\sum_i \bar g_i \circ \bar f_i = p\td \circ \text{Id}_{W\td}\circ i\td = \text{Id}_{W_1\td}$. Therefore $W_1$ is generated by $V$. 
\end{proof}

The following lemma is a standard fact in quantum theory. 
 
\begin{lem}\label{slambda}
Given a composition $\lambda=(\lambda_1,\cdots,\lambda_n)$ of $d$, let $V_{q,\lambda}$ be the subspace of $V_n^{\otimes d}$ generated (as a vector space) by the vectors $v_{i_{\sigma (1)}}\otimes\cdots\otimes v_{i_{\sigma(d)}}$ where $v_{i_1}\otimes\cdots\otimes v_{i_d}=v_1^{\otimes \lambda_1}\otimes\cdots\otimes v_n^{\otimes\lambda_n},$ and $\sigma\in S_d$,
where $\{v_i\}$ is the standard basis of $V_n$. Then $V_{q,\lambda}$ is a direct summand of $V_n^{\otimes d}$ as a $\mathcal{B}_d$-module. 
\end{lem}

\begin{proof}
The generator $T_j$ of the braid group maps the vector $v=\cdots \otimes v_{i_j}\otimes v_{i_{j+1}} \otimes \cdots$ to a linear combination of $v$ and $v'=\cdots \otimes v_{i_{j+1}}\otimes v_{i_{j}} \otimes \cdots$. See equation (\ref{def:Rmatrix}). It follows that $T_i$ leaves $V_{q,\lambda}$ invariant. Since this is true for all $\lambda$, the submodule $V_{q, \lambda}$ is a direct summand of $V_n^{d}$ as $\mathcal{B}_{d}$-modules. 
\end{proof}

\begin{rem}
As a vector space, $V_{q,\lambda}$ is the same as the permutation module $M_\lambda$ in $V_n^{\otimes d}$ viewed as an $S_d$-module.
\end{rem}

If $W\in \Gamma_{q,e}^d\mathcal V$ is of the form $V_n^{\otimes e}$, then we can formulate an explicit sufficient condition for what generates $W$.

\begin{prop}\label{standardehecke}
The $e$-Hecke pair $V_{m}^{\otimes e}\in\Gamma_{q,e}^d\mathcal V$ generates $V_n^{\otimes e}$ for any $n\in \N$  if $m\geq de$.
\end{prop}
\begin{proof}

First let $e=1$. By Proposition \ref{equicon}, it is enough to prove that any $\mathcal{B}_d$-indecomposable summand of $V_n^{\otimes d}$ is a $\mathcal{B}_d$-summand of $V_m^{\otimes d}$. Let $M$ be an indecomposable summand of $V_n^{\otimes d}$ as $\mathcal{B}_d$-modules. It follows from Lemma \ref{slambda} that $M$ is an indecomposable summand of $V_{q,\lambda}$ for some $\lambda$. If $m \geq d$, then $V_{q,\lambda}$ is isomorphic to a summand of $V_m^{\otimes d}$. The conclusion follows.    

For general $e$, consider the $\mathcal{B}_{de}$-modules $V_n^{\otimes de}$ and $V_m^{\otimes de}$. The proof of the case $e=1$ (where $d$ is replaced by $de$), implies that there are $\mathcal{B}_{de}$-maps \[f_i:V_n^{\otimes de}\to V_m^{\otimes de},\ g_i:V_m^{\otimes de}\to V_n^{\otimes de}\] such that $\sum_i g_i\circ f_i=\Id_{V_n^{\otimes de}}$, provided that $m\geq ed$. Consider the subgroup, call it $\mathcal{B}_{d,e}$, of $\mathcal{B}_{de}$ generated by $T_{w_1},\cdots,T_{w_{d-1}}$, where $w_i$ are as in \eqref{w}. 

The action of $T_{w_i} \in \mathcal{B}_{d,e}$ on $V_n^{\otimes de}$ is the same as the action of $T_i \in \mathcal{B}_d$ on $(V_n^{\otimes e})^{\otimes d} = V_n^{\otimes de}$ via $(R_{V_n^{\otimes e}})_{i,i+1}$.
The $\mathcal{B}_{de}$-maps $f_i$ and $g_i$ can then be viewed as $\mathcal{B}_d$-maps 
\[f_i:(V_n^{\otimes e})^{\otimes d}\to (V_m^{\otimes e})^{\otimes d},\ g_i:(V_m^{\otimes e})^{\otimes d}\to (V_n^{\otimes e})^{\otimes d}\]
with $\sum_i g_i\circ f_i=\Id$.
This gives the desired result. 
\end{proof}

Define $(V_{m}^{\otimes e})^{(i)}$ to be a (numbered) copy of $V_{m}^{\otimes e}$.

\begin{prop}\label{standardeheckedirectsum}
The $e$-Hecke pair $\oplus_{i=1}^d (V_{m}^{\otimes e})^{(i)}  \in\Gamma_{q,e}^d\mathcal V$ generates $\oplus_{i=1}^N (V_n^{\otimes e})^{(i)}$ for any $n, N \in \N$  if $m\geq de$.
\end{prop}
\begin{proof}
Given a composition $\lambda=(\lambda_1,\cdots,\lambda_N)$ of $d$, let $U_{n,\lambda} \subset (\oplus_{i=1}^N (V_n^{\otimes e})^{(i)})^{\otimes d}$ be the direct sum 
\[ \bigoplus_{\sigma \in S_d} (V_n^{\otimes e})^{(i_{\sigma(1)})} \otimes (V_n^{\otimes e})^{(i_{\sigma(2)})} \otimes \cdots \otimes (V_n^{\otimes e})^{(i_{\sigma(d)})}\]
where 
$(V_n^{\otimes e})^{(i_{1})} \otimes \cdots \otimes (V_n^{\otimes e})^{(i_{d})} :=
((V_n^{\otimes e})^{(1)})^{\otimes \lambda_1} \otimes  \cdots \otimes ((V_n^{\otimes e})^{(N)})^{\otimes \lambda_N}$.
The space $U_{n, \lambda}$ is a direct summand of $ (\oplus_{i=1}^N (V_n^{\otimes e})^{(i)})^{\otimes d}$ as a $\mathcal{B}_d$ module and $ (\oplus_{i=1}^N (V_n^{\otimes e})^{(i)})^{\otimes d}$ is a direct sum of $U_{n, \lambda}$ over $\lambda$. Let $\bar \lambda = (\bar \lambda_1, \cdots \bar \lambda_{\bar N})$ be the composition $\lambda$ with all the $0$'s removed, then $\bar N \leq d$. If $\bar N < d$, add $0$'s at the end of $\bar \lambda$ such that the number of entries in $\bar \lambda$ is exactly $d$. 

Define $U_{m, \bar \lambda} \subset (\oplus_{i=1}^d (V_m^{\otimes e})^{(i)})^{\otimes d}$ in much the same way we defined $U_{n, \lambda}$ above. Then by an argument similar to the proof of Proposition \ref{standardehecke} it follows that $U_{m,\bar \lambda}$ generates $U_{n,\lambda}$. By Proposition \ref{equicon} we are done.  
\end{proof}

The special cases in the Propositions \ref{standardehecke} and \ref{standardeheckedirectsum} are enough to guarantee a projective generator in the semisimple situation. Denote $\oplus_{i=1}^d (V_{m}^{\otimes e})^{(i)}  \in\Gamma_{q,e}^d\mathcal V$ by $W^e_{m,d}$. 

\begin{thm}\label{thm:finitegeneration}
Suppose $q$ is not a root of unity and $\operatorname{char}(k)=0$. Then the functor $\Gamma^{d,W^e_{m,d}}_{q,e}$ is a projective generator in $\pde$ if $m\geq ed$.
\end{thm}
\begin{proof}
By Corollary \ref{gammapg}, it is enough to show that any $W\in\Gamma^d_{q,e}\mathcal V$ is generated by $W^e_{m,d}$. But by Propositions \ref{semisimple} and \ref{prop:subcomodule}, $W$ is a direct sum of direct summands of some $V_n^{\otimes e}$,which means it is a direct summand of $\oplus_{i=1}^N (V_n^{\otimes e})^{i}$. So Proposition \ref{standardeheckedirectsum} and Lemma \ref{karoubi} shows that $W^e_{m,d}$ generates $W$.

The fact that $\Gamma^{d,W^e_{m,d}}_{q,e}$ is projective is just Proposition \ref{prop:projective}. 
\end{proof} 
 
 \begin{cor}\label{cor:equivalenceofcategories}
The evaluation functor $\mathcal{P}^d_{q,e} \to \operatorname{mod}(S_q(W^e_{m,d},W^e_{m,d};d))$ is an equivalence of categories for $q$ not a root of unity, $\operatorname{char}(k) = 0$ and $m \geq de$. It follows that $S_q(W^e_{m,d},W^e_{m,d};d)$ and $S_q(W^e_{n,d},W^e_{n,d};d)$ are Morita equivalent when $m,n \geq de$.  
\end{cor}

The main motivation behind our construction is composition of polynomial functors. For this to work in the greatest generality, we require direct sum of indecomposable $e$-Hecke pairs to be in the domain. To see this consider the simplest case possible, when $e=1, d=1$. Consider the polynomial functor mapping the $1$-Hecke pair $V_1 \mapsto V_1 \oplus V_1$. This defines the functor since $V_1$ is a projective generator by Theorem \ref{thm:finitegeneration}. It follows that $V_1 \oplus V_1$ must be in the domain if we wish to compose any two polynomial functors (subject to constraints on $d$ and $e$). 

The category $\pde$ is equivalent to the module category of a generalized Schur algebra. We now present a category who has the same property, but whose Schur algebra is slightly simpler.   

Recall the category $\mathcal{P}_{q,e}^{\circ, d}$ from Definition \ref{def:qpfcirc}. Each quantum polynomial functor in $\pde$ can then be restricted to an object of $\mathcal{P}_{q,e}^{\circ, d}$. The added condition on the domain makes it so that composition is not possible in $\mathcal{P}_{q,e}^{\circ, d}$ (in the simplest case presented above, because $V_1 \oplus V_1$ is not part of the domain of $\mathcal{P}_{q,1}^{\circ, d}$). However this condition allows one to prove the existence of a simpler projective generator. 

\begin{thm}\label{thm:finitegenerationindecomposable}
Suppose $q$ is not a root of unity and $\operatorname{char}(k)=0$. Then the functor $\Gamma^{d,V_m^{\otimes e}}_{q,e}$ is a projective generator in $\mathcal{P}^{\circ, d}_{q,e}$ if $m\geq ed$.
\end{thm}
 \begin{proof}
 The proof follows along the same lines as the proof of Theorem \ref{thm:finitegeneration}, only now we require the simpler Proposition \ref{standardehecke} (instead of Proposition \ref{standardeheckedirectsum}) because of the extra condition on the domain. 
 \end{proof}
 
 \begin{cor}\label{cor:equivalenceofcategoriesindecomposable}
The evaluation functor $\mathcal{P}^{\circ, d}_{q,e} \to \operatorname{mod}(S_q(V_n^{\otimes e},V_n^{\otimes e},d))$ is an equivalence of categories for $q$ not a root of unity and $\operatorname{char}(k) = 0$. It follows that $S_q(V_n^{\otimes e},V_n^{\otimes e},d)$ and $S_q(V_m^{\otimes e},V_m^{\otimes e},d)$ are Morita equivalent when $m,n \geq de$.  
\end{cor}
 
The category $\mathcal{P}^{\circ, d}_{q,1}$ is equivalent to the category studied by Hong and Yacobi in \cite{HongYacobi}. The category $\mathcal P^d_{q,1}$ is strictly greater. %In \cite{HongYacobi}, the domain is made out of defining representations $V_n$. 
In order to be able to define composition one needs to consider higher degree comodules in the domain (which produces $\mathcal{P}_{q,e}^{\circ, d}$) and then also consider direct sums (which produces $\pde$). 

\begin{rem}\label{eqtoHY}
Setting $e=1$ in Theorem \ref{thm:finitegenerationindecomposable}, we obtain Theorem 4.7 in Hong and Yacobi \cite{HongYacobi} when $q$ is generic and $\operatorname{char}(k)=0$. But note that our proof for $e=1$ works when $q$ is a root of unity or $\operatorname{char}(k)\neq 0$ with a minor addition which we now explain. Theorem \ref{thm:finitegenerationindecomposable} depends on Corollary \ref{gammapg}, Proposition \ref{standardehecke} and Lemma \ref{karoubi} which are true regardless if $q$ is a root of unity or not and if $\operatorname{char}(k)$ is $0$ or not. It also depends on Propositions \ref{semisimple} and \ref{prop:subcomodule}, which are not true in general for $q$ a root of unity or $\operatorname{char}(k) = 0$. However, when $e=1$, Propositions \ref{semisimple} and \ref{prop:subcomodule} hold for $q$ a root of unity or $\operatorname{char}(k) \neq 0$ because $1$-Hecke pairs are just direct sums of the defining comodule $V_n$ for any $n$. Thus we obtain Theorem 4.7 in \cite{HongYacobi} with no restrictions on $q$ or the characteristic of  $k$. 
\end{rem}

\begin{rem}\label{rem:SWduality}
We can think of the generalized Schur algebras $S_q(V_n^{\otimes e}, V_n^{\otimes e};d)$ for generic $q$ as follows. Quantum Schur-Weyl duality (due to Jimbo) says that there is a commuting action of the Hecke algebra $\mathcal{H}_d$ and the $q$-Schur algebra $S_q(V_n,V_n;d)$ on the space $V_n^{\otimes d}$ and these two actions satisfy a double centralizer property. The $e$-Hecke algebra $\mathcal{H}_{d,e} \subset \mathcal{H}_{de}$ acts on $(V_n^{\otimes e})^{\otimes d}$ as explained before. Then $S_q(V_n^{\otimes e}, V_n^{\otimes e};d) \supset S_q(V_n, V_n; de)$ is conjecturally the object that makes the following diagram satisfy a double centralizer property on both rows.
\begin{center}
\begin{tikzcd}[
  ar symbol/.style = {draw=none,"#1" description,sloped},
  isomorphic/.style = {ar symbol={\cong}},
  subset/.style = {ar symbol={\subset}},
  supset/.style = {ar symbol={\supset}},
  equals/.style = {ar symbol={=}},
  ]
  S_q(V_n^{\otimes e}, V_n^{\otimes e};d) & \curvearrowright & (V_n^{\otimes e})^{\otimes d} &\curvearrowleft  & \mathcal{H}_{d,e}  \\
  S_q(n, n;de) \ar[u,subset] & \curvearrowright & V_n^{\otimes de}  \ar[u,isomorphic] & \curvearrowleft & \mathcal{H}_{de} \ar[u,supset] \\
\end{tikzcd}
\end{center}
%The upper level satisfies, by the definition of $S_q(V_n^{\otimes e}, V_n^{\otimes e};d)$, one half of the double centralizer property.  
\end{rem}

\section{The category $\widetilde{\pd}$}\label{sec:where}

We now define a category that ``lives inbetween'' the category $\pde$ for any $e$ and the category $\pd$. 
Let $d$ be a positive integer. The quantum divided power category
$\Gamma^d_{q} \mathcal{V}$ is the category with objects formal finite direct sums $\oplus_i V_i$ where each $V_i$ is an $e$-Hecke for some $e$. 
The morphisms are defined on homogeneous objects as
follows: \[ \text{Hom}_{\Gamma^d_{q}} (V_i, W_j):= \text{Hom}_{\mathcal{B}_d}
(V_i^{\otimes d}, W_j^{\otimes d}). \]
The Hom extends naturally to all objects via the formula 
\[\text{Hom}_{\Gamma^d_{q}} \Big(\bigoplus_i V_i, \bigoplus_j W_j \Big) = \bigoplus_{i,j} \text{Hom}_{\Gamma^d_{q}}(V_i, W_j).\] 

\begin{defn}\label{def:qpftilde}
The category $\widetilde{\pd}$ is the category of of linear functors
\[ F : \Gamma^d_{q} \mathcal{V} \to \Gamma^1_q\mathcal{V}. \]
Morphisms are natural transformations of functors. Note that $\Gamma^1_q\mathcal V$ is equivalent to $\mathcal V$.
\end{defn}

\begin{rem}
Given a linear functor $F : \Gamma^d_{q} \mathcal{V} \to \Gamma^1\mathcal{V} $, denote by $F_e$ its restriction to $\Gamma^d_{q,e} \mathcal{V}$. Then it is not hard to see that $F_e \in \pde$.  
\end{rem}

We can define a composition on $\widetilde{\pd}$ similarly to how we defined composition between a functor in $ \mathcal{P}^{d_1}_{q,d_2 e}$ and a functor in $\mathcal{P}^{d_2}_{q, e}$.  

It is interesting to note that when $q = 1$, the category $\widetilde{\pd}$ becomes the category of classical polynomial functors $\mathcal{P}^d$ due to Friedlander and Suslin. Therefore when $q=1$, the category $\widetilde{\pd}$ and the categories $\pde$ for any $e$ are all equivalent. We show that for generic $q$, $\widetilde{\pd}$ is not equivalent to $\pde$ for any $e$. In fact in the quantum case $\widetilde{\pd}$ is ``closer'' to $\pd$ than to $\pde$ since both $\widetilde{\pd}$ and $\pd$ are $not$ finitely generated while $\pde$ is for $q$ is generic. This is another reason why the theory of quantum polynomial functors is richer in the quantum case than it is in the classical case. 

\begin{prop}\label{pro:notfinitegeneration}
Suppose $q$ is not a root of unity. Then the category $\widetilde{\pd}$ is not finitely generated. In particular, $\widetilde{\pd}$ is not equivalent to $\pde$ for any $e$.
\end{prop}
\begin{proof}
Consider the sequence of objects $k^{\otimes f}$ in $\Gamma^d_q\mathcal V$, where $f \in \N$ and $k=(k,q)$ is the trivial $A_q(1,1)$-comodule of degree $1$. Then each $(k^{\otimes f})^{\otimes d}$ is a one-dimensional $\mathcal{B}_d$-module on which each $T_i$ acts as multiplication by $q^{l(w_i)}=q^{f^2}$. These form an infinite collection of irreducibles $\mathcal{B}_d$-modules since $q$ is not a root of unity. 
Take $F=\bigotimes^d$ in \ref{def:qpftilde}; we show that no $V$ makes the evaluation map 
\[V^{\otimes d}\otimes \Gamma_{q,e}^{d,V}(k^{\otimes f})=V^{\otimes d}\otimes \Hom_{\mathcal{B}_d}(V^{\otimes d},(k^{\otimes f})^{\otimes d})\to (k^{\otimes f})^{\otimes d}\] 
surjective for all $f$.
If there was such $V$, then $V^{\otimes d}$ (as $\mathcal{B}_d$-module) would contain all the $k^{\otimes f}$ above.
This is impossible since $V$ is finite dimensional. 
\end{proof}

\begin{rem}
Note that the proof makes use of the normalization of the $R$-matrix $R_n$ defined in equation \eqref{def:Rmatrix}. In particular, we use that $R_n (v_1 \otimes v_1) = q v_1 \otimes v_1$. For a different normalization (for example where $R_n (v_1 \otimes v_1) = v_1 \otimes v_1$) the proof above doesn't work. The result stays true, while the argument becomes computationally more complicated. One needs to look at $R_2$ (which will have an eigenvalue $-q^{-2}$ for the normalization mentioned above) and modify the proof of Proposition \ref{pro:notfinitegeneration} accordingly.
\end{rem}

\section{Remarks on quantum polynomial functors when $q$ is a root of unity}\label{sec:rootsofunity}

Let $q$ be an $l$-th root of unity where $l>1$ is an odd integer.

\subsection{Representability}

The proof of Theorems \ref{thm:finitegeneration} and \ref{thm:finitegenerationindecomposable} are not valid in this case, since an indecomposable $e$-Hecke pair $W\in \operatorname{comod}(A_q(n,n))$ is not necessarily a direct sum of summands of $V_n^{\otimes e}$.
But it can still be true that $\pde$ is finitely generated, hence equivalent to the module category of a finite dimensional algebra. 
For example in the classical case ($q=1$) when the field $k$ has characteristic $p$, the category of polynomial functors is not semisimple, but it does have a projective generator just as in the case when the characteristic of the field is $0$. %@@!!I don't like this and want to see what I wrote before but I am offline..@@@ !!! This is something you wrote, you can change it as you see fit. V 
In this section we present some remarks on the case when $q$ is a root of unity not equal to $1$.

First recall that it is only at the last step of the proof that we use the semisimplicity. In particular, Corollary \ref{gammapg}, Lemma \ref{karoubi} and Proposition \ref{equicon} are valid when $q$ is non-generic. We summarize them as a separate statement.

\begin{prop}\label{fgcond}
Let $V$ be an $e$-Hecke pair. Assume for any $e$-Hecke pair $W$, every indecomposable $\mathcal{B}_d$-summand of $W\td$ is isomorphic, as a $\mathcal{B}_d$-module, to a direct summand of $V\td$. Then the category $\pde$ has a (finite) projective generator $\Gamma^{d,V}_{q,e}$.
\end{prop}
Note that the set of divided powers generates $\pde$, thus if $\pde$ is finitely generated then one can find a functor of the form $\Gamma^{d,V}_{q,e}$ which is a projective generator.

The condition in Proposition \ref{fgcond} is reduced to an elementary statement about Jordan block decomposition of $R$-matrices if $d=2$.

\begin{prop}\label{prop:jordanblock}
%Consider the R-matrices for standard $e$-Hecke pairs $V_n^{\otimes e}$ for each $n$. 
Let $n\in \N$, and let $U$ be an indecomposable $A_q(m,m)$-comodule of degree $e$. Denote the Jordan blocks of $R_{V_n^{\otimes e}}$ by $B(n_i,a_i)$, where $i$ runs through some finite index set $I$, $n_i\in \mathbb N$ is the rank of the block and $a_i\in k$ is the generalized eigenvalue of the block. 
Similarly, name the Jordan blocks of $R_U$ by $B(m_j,b_j)$, where $j\in J$. 

Then the identity map on $U\otimes U$ factors through $V_n^{\otimes e}\otimes V_n^{\otimes e}$ as a $\mathcal{B}_2$-map if and only if for each $j\in J$, there exists $i\in I$ such that $a_i=b_j$ and $n_i=m_j$.  
\end{prop}
\begin{proof}
It is enough to factor the identity on each Jordan block of $U$. But we can do it by embedding $B(m_j,b_j)$ onto the block $B(n_i,a_i)$.
\end{proof}

Therefore understanding the Jordon block decomposition of $R_U$ would allow one to prove representability for $\mathcal{P}_{q,e}^{\circ, d}$. A statement similar to Proposition \ref{prop:jordanblock}, where one replaces the indecomposable $U$ by any $e$-Hecke pair $U$ can also be proven.  

%@@see if you like the following@@ We do not know if $U$ is isomorphic to a subquotient of the $e$-Hecke pair of the form $V_n^e$. We remark that we formulate the Proposition partly because it is improved to give a criterion appropriate for computer calculation if this is the case.!!! I like it, though maybe we can write a remark later when you come back. 

\begin{rem}
The condition in Proposition \ref{prop:jordanblock} is trivially true if the $R$-matrix of any $e$-Hecke pair is diagonalizable. This is the case when $q$ is not a root of unity, but it is not true when $q$ is a root of unity. For example $R_q$ is not diagonalizable when $q = \pm i$. If we look only at $1$-Hecke pairs, these are the only values for $q$ where $R_V$ is not diagonalizable. For a general $e$-Hecke pair $V$, we expect there are other roots of unity for which $R_V$ is not diagonalizable.  
\end{rem}

\begin{rem}
It is possible that the functor cohomology for quantum polynomial functors agrees with the corresponding quantum group cohomology even if the category $\pde$ does not have a finite generator. A similar approach is found in Suslin's appendix in \cite{FFSS}.
\end{rem}

\subsection{An additional structure}
An important feature when $q$ is a root of unity is the existence of the Frobenius twist. 
When $e=1$, this structure comes from the $q$-Schur algebra.
Let us write this in terms of polynomial functors.
This is explained in a previous version of \cite{HongYacobi}, which we now repeat. The algebra map $\operatorname{Fr_{n,m}}:A_1(n,m)\to A_q(n,m)$ 
defined on the standard generators by \[x_{ij} \mapsto x_{ij}^l\] is also a coalgebra map. 
Using this map, we can define a functor
\[(-)^{[1]}:\mathcal{P}^d\to \mathcal{P}^{ld}_{q,1}.\]
We call it the Frobenius twist.
(We remind the reader that the classical polynomial functor category $\mathcal P^d$ can be viewed as $\mathcal P^d_{1,e}$ for any $e$, and that $\mathcal P^{ld}_{q,1}$ is equivalent to the category $\mathcal P^{ld}_q$ of Hong-Yacobi \cite{HongYacobi}. See Remark \ref{eqtoHY}.)
Given $F\in\mathcal{P}^d$, its Frobenius twist $F^{[1]}$ is defined to be a functor from $\Gamma^d_{q,1}\mathcal V$ to $\Gamma^1_{q,d}\mathcal V\cong \mathcal V$ that sends a $1$-Hecke pair $V$ to the vector space $F(V)$ (forgetting the $1$-Hecke structure, $V$ is viewed as an object in $\Gamma^d\mathcal{V}$). 
To define what $F^{[1]}$ does to morphisms, it is enough to specify the map
\[(F^{[1]})''_{V,W}:F(V)\to F(W)\otimes A_q(W,V).\]
Note that $V$ and $W$ are $1$-Hecke pairs, hence direct sum of standard $A_q(n,n)$-comodules and $A_q(m,m)$-comodules, respectively, for certain $n, m$. We define the map $\operatorname{Fr}: A_1(W,V)\to A_q(W,V)$ as a straightforward generalization of $\operatorname{Fr}_{n,m}$. It maps 
\[x_{ij} \in \bigoplus A_1(V_n, V_m) \cong A_1(V,W) \mapsto x_{ij}^l \in \bigoplus A_q(V_n, V_m) \cong A_q(V,W).\] 
This extends to a bialgebra map. 
Now we define the map $(F^{[1]})''_{V,W}$; it is the composition 
\[(1_{F(W)}\otimes \operatorname{Fr}) \circ F''_{V,W}:F(V)\to F(W)\otimes A_1(W,V)\to F(W)\otimes A_q(W,V).\]
It is not hard to see that $(F^{[1]})''_{V,W}$ satisfies the properties needed to make $F^{[1]}$ into a quantum polynomial functor.

Consider $I^{[1]}\in \mathcal{P}_{q,1}^l$ where $I\in\mathcal{P}^1$ is the identity polynomial $\mathcal{V}=\Gamma^1\mathcal{V}\to\mathcal{V}$.
Then for any $d$, the composition
\[-\circ - :\mathcal{P}^d_{q,l}\otimes \mathcal P^l_{q,1}\to \mathcal P^{dl}_{q,1}\]
induces the functor
\[-\circ I^{[1]}:\mathcal{P}^d_{q,l}\to \mathcal P^{dl}_{q,1}.\]
Since the Frobenius $I^{[1]}$ lives only in the category $\mathcal{P}_{q,1}^l$, we cannot compose the Frobenius with itself multiple times (as in the classical case). 

One may ask if we can define a class of objects $I^{[1]}_e\in\mathcal P^l_{q,e}$ for all $e$ which are reasonable analogues of the Frobenius. 
This will supply higher Frobenius twists using composition: \[I^{[r]}_e:=I^{[1]}_{l^{r-1}e}\circ \cdots\circ I^{[1]}_e.\] Precomposing or postcomposing $I^{[r]}_e$ provides functors between various polynomial functor categories.

However, $I^{[1]}_e$ does not come from the structure of the $q$-Schur algebras as in the $e=1$ case; if $e>1$ we do not have an analogous coalgebra map \[A_1,(V,W)\to A_q(V,W)\] where $V,W$ are $e$-Hecke pairs. 
Thus we must take a different approach to define $I^{[1]}_e$. 
One may try to define the Frobenius as the cohomology of certain complexes that are of interest by themselves.

Classically (by this we mean $q=1$ and $\text{char}(k)=p$),
we have the following exact sequence of polynomial functors:
\begin{equation}\label{frob1}
0\to I^{[1]}\to S^p\to \Gamma^p\to I^{[1]}\to 0.
\end{equation}
That is, one can define the Frobenius polynomial $I^{[1]}\in \mathcal P^p_{1,e}$ as either the kernel or the cokernel of the middle map in \eqref{frob1}.
Alternatively, 
the following complex, called the $d$-th de Rham complex, has nontrivial cohomology when $p$ divides $d$. 
\begin{equation}\label{deRham}
 0 \to S^d \to S^{d-1} \otimes \Lambda^1 \to S^{d-2} \otimes \Lambda^2 \to \cdots \to \Lambda^d \to 0  
\end{equation}
If $d=ap$, its cohomology is given by the Frobenius twist of the $(a-1)p$-th de Rham complex. 

One can try to quantize these complexes. That gives a way to define $I^{[1]}_e$.
Note that all the objects in \eqref{frob1} and \eqref{deRham} except the Frobenius are defined in $\pde$ for all $e$.
Also note that the quantum symmetric powers and quantum divided powers usually behave very differently when we move away from the degree $1$ case to the degree $e$ case. For example, the dimension of $S^d_q(V)$ depends on the $A_q(n,n)$-comodule structure of $V$ and not only on the dimension of $V$ (this phenomenon is investigated in \cite{BZ}). Therefore the dimension of $I^{[1]}_e(V)$ for an $e$-Hecke pair $V$ might be different from the dimension of $V$, in contrary to the classical case. In particular, we cannot obtain $I^{[1]}_e(V)$ from the underlying space of $V$ by just twisting the module structure. 
%It is not clear whether we obtain an isomorphic object $I^{[1]}_e$ from these (at least three) different definitions.

We note that the middle map in the exact sequence \eqref{frob1} can be quantized into a map between polynomial fucntors acting on $1$-Hecke pairs and for $q$ a root of unity (and $\text{char}(k)=0$) we can define the Frobenius functor as either the kernel or the image of that map. This definition of the Frobenius coincides with definition of $I^{[1]}$ via the Frobenius twist. This a sign that a homological approach to defining the Frobenius in the quantum case is worth investigating. 

This discussion can be the starting point of further investigations. One can try to define the quantum Frobenius twist as mentioned above, and try to understand its properties. Then, one can try to understand its role in quantum theory, or its uses in cohomology theory. %One can also study differences in the quantum definition versus the classical definition in the spirit of \cite{BZ} .

\bibliography{bib.bib}

\begin{thebibliography}{FFSS99}

\bibitem[BZ08]{BZ}
Arkady Berenstein and Sebastian Zwicknagl.
\newblock Braided symmetric and exterior algebras.
\newblock {\em Trans. Amer. Math. Soc.}, 360(7):3429--3472, 2008.

\bibitem[Cha15]{Chaluntwist}
Marcin Chalupnik.
\newblock Derived {K}an extension for strict polynomial functors.
\newblock {\em Int. Math. Res. Not. IMRN}, (20):10017--10040, 2015.

\bibitem[CP94]{CP}
Vyjayanthi Chari and Andrew Pressley.
\newblock {\em A guide to quantum groups}.
\newblock Cambridge University Press, Cambridge, 1994.

\bibitem[DJ89]{DipperJames}
Richard Dipper and Gordon James.
\newblock The {$q$}-{S}chur algebra.
\newblock {\em Proc. London Math. Soc. (3)}, 59(1):23--50, 1989.

\bibitem[FFSS99]{FFSS}
Vincent Franjou, Eric~M. Friedlander, Alexander Scorichenko, and Andrei Suslin.
\newblock General linear and functor cohomology over finite fields.
\newblock {\em Ann. of Math. (2)}, 150(2):663--728, 1999.

\bibitem[FRT88]{FRT}
L.~D. Faddeev, N.~Yu. Reshetikhin, and L.~A. Takhtajan.
\newblock Quantization of {L}ie groups and {L}ie algebras.
\newblock In {\em Algebraic analysis, {V}ol. {I}}, pages 129--139. Academic
  Press, Boston, MA, 1988.

\bibitem[FS97]{FS}
Eric~M. Friedlander and Andrei Suslin.
\newblock Cohomology of finite group schemes over a field.
\newblock {\em Invent. Math.}, 127(2):209--270, 1997.

\bibitem[GU89]{GU}
Akihiko Gyoja and Katsuhiro Uno.
\newblock On the semisimplicity of {H}ecke algebras.
\newblock {\em J. Math. Soc. Japan}, 41(1):75--79, 1989.

\bibitem[GW98]{GW}
Roe Goodman and Nolan~R. Wallach.
\newblock {\em Representations and invariants of the classical groups},
  volume~68 of {\em Encyclopedia of Mathematics and its Applications}.
\newblock Cambridge University Press, Cambridge, 1998.

\bibitem[HTY14]{HTY}
Jiuzu Hong, Antoine Touz\'e, and Oded Yacobi.
\newblock Polynomial functors and categorifications of {F}ock space.
\newblock In {\em Symmetry: representation theory and its applications}, volume
  257 of {\em Progr. Math.}, pages 327--352. Birkh\"auser/Springer, New York,
  2014.

\bibitem[HY13]{HY2}
Jiuzu Hong and Oded Yacobi.
\newblock Polynomial functors and categorifications of {F}ock space {II}.
\newblock {\em Adv. Math.}, 237:360--403, 2013.

\bibitem[HY17]{HongYacobi}
Jiuzu Hong and Oded Yacobi.
\newblock Quantum polynomial functors.
\newblock {\em J. Algebra}, 479:326--367, 2017.

\bibitem[Jan96]{quantumJ}
Jens~Carsten Jantzen.
\newblock {\em Lectures on quantum groups}, volume~6 of {\em Graduate Studies
  in Mathematics}.
\newblock American Mathematical Society, Providence, RI, 1996.

\bibitem[LR97]{Lambe}
Larry~A. Lambe and David~E. Radford.
\newblock {\em Introduction to the quantum {Y}ang-{B}axter equation and quantum
  groups: an algebraic approach}, volume 423 of {\em Mathematics and its
  Applications}.
\newblock Kluwer Academic Publishers, Dordrecht, 1997.

\bibitem[PW91]{PW}
Brian Parshall and Jian~Pan Wang.
\newblock Quantum linear groups.
\newblock {\em Mem. Amer. Math. Soc.}, 89(439):vi+157, 1991.

\bibitem[Tak02]{Takeuchi}
Mitsuhiro Takeuchi.
\newblock A short course on quantum matrices.
\newblock In {\em New directions in {H}opf algebras}, volume~43 of {\em Math.
  Sci. Res. Inst. Publ.}, pages 383--435. Cambridge Univ. Press, Cambridge,
  2002.
\newblock Notes taken by Bernd Str\"uber.

\bibitem[Tou10]{TouzeAIM}
Antoine Touz\'e.
\newblock Cohomology of classical algebraic groups from the functorial
  viewpoint.
\newblock {\em Adv. Math.}, 225(1):33--68, 2010.

\bibitem[Tou13]{Touzeuntwisting}
A.~Touz\'e.
\newblock A construction of the universal classes for algebraic groups with the
  twisting spectral sequence.
\newblock {\em Transform. Groups}, 18(2):539--556, 2013.

\end{thebibliography}
\bibliographystyle{alpha}  

\end{document}